\documentclass[12pt, twoside]{article}
\usepackage{mathrsfs}
\usepackage{bbm}
\usepackage{cite}
\usepackage{amsfonts}
\usepackage{amsmath,amssymb}
\usepackage{multicol}

\usepackage{latexsym}
\usepackage{graphicx}
\usepackage{dcolumn}
\usepackage{bm}
\usepackage{fancyhdr}
\usepackage{mathrsfs}
\usepackage{wasysym}
\usepackage{epsf}
\usepackage{epsfig}
\usepackage{psfrag}
\usepackage{ifpdf}
\usepackage[
  bookmarks=true,
  bookmarksopen=true,
  breaklinks=true,
  colorlinks=true,
  linkcolor=blue,anchorcolor=blue,
  citecolor=blue,filecolor=blue,
  menucolor=blue,
  urlcolor=blue]{hyperref}

 \numberwithin{equation}{section}
\newtheorem{theorem}{Theorem}[section]
\newtheorem{lemma}[theorem]{Lemma}

\newtheorem{remark}[theorem]{Remark}

\newenvironment{proof}[1][Proof]{\noindent\textbf{#1.} }{\hfill $\Box$}

\allowdisplaybreaks

\allowdisplaybreaks \numberwithin{equation}{section}
\makeatletter
\begin{document}

\author{Vladimir Georgiev{$^a$$^b$$^c$} ~Yuan Li{$^d$$^a$}\footnote{Corresponding author}~
\renewcommand\thefootnote{}
\footnote{{E-mail addresses}:  georgiev@dm.unipi.it, liyuan2014@lzu.edu.cn (Y. Li)}
\setcounter{footnote}{0}\\
 \small ${^a}$Dipartimento di Matematica, Universit\`{a} di Pisa, Largo B. Pontecorvo 5, 56100 Pisa, Italy
\\ \small ${^b}$ Faculty of Science and Engineering, Waseda University, 3-4-1, Okubo, Shinjuku-ku, \\ \small Tokyo 169-8555, Japan
\\ \small ${^c}$ IMI¨CBAS, Acad. Georgi Bonchev Str., Block 8, 1113 Sofia, Bulgaria
\\ \small ${^d}$ School of Mathematics and Statistics, Lanzhou
 University, Lanzhou, 730000, PR China}
\date{}
\title{\textbf{Blowup dynamics for Mass Critical Half-wave equation in 3D}}
\maketitle
\begin{abstract}
We consider the half-wave equation $i u_t=Du-|u|^{\frac{2}{3}}u$ in three dimension and in the mass critical. For initial data $u(t_0,x)=u_0(x)\in H^{1/2}_{rad}(\mathbb{R}^3)$ with radial symmetry, we construct a new class of minimal mass blowup solutions with the blow up rate $\|D^{\frac{1}{2}}u(t)\|_2\sim\frac{C(u_0)}{|t|^{\frac{1}{4}}}$ as $t\rightarrow0^-$.


\end{abstract}

\section{Introduction and Main Results}
\label{section:1}
In this paper, we consider the half-wave equation in three dimension
\begin{equation}\label{equ-1-hf-3D}
\begin{cases}
i\partial_tu=Du-|u|^{\frac{2}{3}}u,\\
u(t_0,x)=u_0(x),\ u:I\times\mathbb{R}^3\rightarrow\mathbb{C}.
\end{cases}
\end{equation}
Here, $I\subset\mathbb{R}$ is an interval containing the initial time $t_0\in\mathbb{R}$, and
\begin{align*}
\widehat{(Df)}(\xi)=|\xi|\hat{f}(\xi)
\end{align*}
denotes the first-order nonlocal fractional derivative.

Let us mention that nonlinear half-wave equation has recently attracted some attention in the area of dispersive nonlinear PDE. Evolution problems  like \eqref{equ-1-hf-3D} arise in various physical settings, which include equations range from turbulence phenomena, wave propagation, continuum limits of lattice system and models for gravitational collapse in astrophysics \cite{Ionescu2014,frank-lenzmann2013,eckhaus1983,majda1997,majda2001,Weinstein1987,k-lenzmann2013}. We also refer to \cite{Elgart2007,FJL2007,klein2014,cho2013} and the references therein for the background of the fractional Schr\"{o}dinger model in mathematics and physics.

For equation \eqref{equ-1-hf-3D}, the quantities of charge $M(u)$ and energy $E(u)$ given by
\begin{align}\label{mass-hf-3D}
\text{Mass}\ \ M(u)&=\int_{\mathbb{R}^3}|u(t,x)|^2dx,\\
\label{momentum-hf-3D}
\text{Momentum}\ \ P(u)&=\int-i\nabla u(t,x)\bar{u}(t,x)dx,\\\label{energy-hf-3D}
\text{Energy}\ \ E(u)&=\frac{1}{2}\int_{\mathbb{R}^3}\bar{u}(t,x)Du(t,x)dx
-\frac{3}{8}\int_{\mathbb{R}^3}|u(t,x)|^{\frac{8}{3}}dx
\end{align}
are conserved. The equation \eqref{equ-1-hf-3D} also has the following symmetry: \begin{align}\notag
u(t,x)\rightarrow\lambda_0^{\frac{3}{2}}u(\lambda_0 t-t_0,\lambda_0x-x_0)e^{i\theta},
\end{align}
for $\lambda_0>0$, $t_0\in\mathbb{R},~\theta\in\mathbb{R}$ and $x_0\in\mathbb{R}^3$.

The Cauchy problem \eqref{equ-1-hf-3D} is $L^2$-critical since the $L^2$-norm is invariant under the scaling rule $u_{\lambda}(x)=\lambda^{3/2}u(\lambda x)$:
\begin{align}\notag
\|u_{\lambda}\|_2=\|u\|_2,\ \text{for all}\ \lambda>0.
\end{align}
From \cite{BGV2018} we known that the Cauchy problem \eqref{equ-1-hf-3D} is locally well-posed in energy space $H^{1/2}(\mathbb{R}^3)$. More precisely, for all $u_0\in H^{1/2}(\mathbb{R}^3)$, there exists a unique solution $u(t)\in C([0,T);H^{1/2}(\mathbb{R}^3))$ to \eqref{equ-1-hf-3D}. Moreover, we have the blowup alternative that if $u(t)$ is the unique solution with its maximal time of existence $t_0<T\leq\infty$, then
\begin{align}\label{Intro-2-3D}
T<+\infty\  \text{implies}\ \lim_{t\rightarrow T^{-}}\|u(t)\|_{H^{1/2}}=+\infty.
\end{align}
A classical criterion of global-in-time existence for $H^{1/2}(\mathbb{R}^3)$ initial data is derived by using the Gagliardo-Nirenberg inequality with best constant
\begin{align}\notag
\|u\|_{L^{5/3}}^{\frac{5}{3}}\leq C_{opt}\|D^{\frac{1}{2}}u\|_2^2\|u\|_2^{\frac{2}{3}},\,\,\text{for}\,\, u\in H^{1/2}(\mathbb{R}^3),
\end{align}
where $C_{opt}=\frac{4}{3}\|Q\|_2^{-\frac{2}{3}}$ and $Q$ is the unique ground state solution to
\begin{equation}\label{equ-s=1-hf-3D}
D Q+Q=|Q|^{\frac{2}{3}}Q,\ Q(x)>0,\ Q(x)\in H^{1/2}(\mathbb{R}^3).
\end{equation}
Note that the existence of this equation follows from standard variational techniques, but the uniqueness of $Q$, which was obtained by Frank, Lenzmann and  Silvestre in \cite{FrankLS2016}. The mass and energy conservation  and the blowup criterion \eqref{Intro-2-3D} implies that initial data $u_0\in H^{1/2}(\mathbb{R}^3)$ with
$\|u_0\|_2<\|Q\|_2$
generate global-in-time solution. In addition,  the study of general half-wave equation attracted a great quantity of attentions; the topics cover over well-posedness, ill-posedness, traveling solitary waves, soliton waves, see \cite{BGLV2019,BGV,BGV2018,Choffrut-2018,Georgiev-Visciglia2016,Lenzmann-Raphael-2018,Li-Zhao-2020,Ozawa-Visciglia2016} and the references therein.

In this paper, we focus on the existence of nondispersive dynamics, we will describe example of such dynamics:

\textbf{Minimal mass blowup solution}.
There is no general criterion for blowup  solutions  for $L^2$-critical or $L^2$-supercritical half-wave equation in $\mathbb{R}^N$. This is still an open problem, which we can see \cite{lenzmann-2016blowup}. However, Krieger, Lenzmann and Rapha\"{e}l \cite{KLR2013} constructed a minimal mass blow-up solutions to the mass critical half-wave equation in one dimension. We also obtained the two-dimensional result, see \cite{Georgiev-Li}.  Now we state our main result.
\begin{theorem}(\textbf{Existence of minimal mass blowup elements})
For all $(E_0,P_0)\in \mathbb{R}_+^*\times\mathbb{R}^3$, there exists $t^*<0$ independent of $E_0$, $P_0$, and a radial minimal mass solution $u\in C^0([t^*,0);H^{1/2}(\mathbb{R}^3))$ of equation \eqref{equ-1-hf-3D} with
\begin{align}\notag
\|u(t)\|_2=\|Q\|_2,\ E(u(t))=E_0,\ P(u(t))=P_0,
\end{align}
which blow up at time $T=0$. More precisely, it holds that
\begin{align}\notag
u(t,x)-\frac{1}{\lambda(t)^{\frac{3}{2}}}Q\left(\frac{x-\alpha(t)}{\lambda(t)}\right)
e^{i\gamma(t)}\rightarrow0\ \text{in}\ L^2(\mathbb{R}^3)\ \text{as}\ t\rightarrow0^-,
\end{align}
where
\begin{align}\notag
\lambda(t)=\lambda^*t^2+\mathcal{O}(t^3),\ \alpha(t)=x_0+\mathcal{O}(t^3),\ \gamma(t)=\frac{1}{\lambda^*|t|}+\mathcal{O}(t),
\end{align}
with some constant $\lambda^*>0$, and the blowup speed is given by:
\begin{align}\notag
\|D^{\frac{1}{2}}u(t)\|_2\sim\frac{C(u_0)}{|t|^{\frac{1}{4}}},
\end{align}
where $C(u_0)>0$ only depend on the initial data $u_0$.
\end{theorem}
\begin{remark}
1. Unlike one or two dimensional case, we need to find some new technicalities to estimate the nonlinear term. In particular,  the following estimate plays an important role in our discuss
\begin{align}
\left\||x|^{-\frac{1}{q_1}}\int_0^te^{i(t-s)D}F(s)ds\right\|_{L^{q_1}(\mathbb{R});L^2(\mathbb{R}^N)}
\leq C\left\||x|^{\frac{1}{q_2}}F\right\|_{L^{q_2}(\mathbb{R});L^2(\mathbb{R}^N)},
\end{align}
where $q_1\in[2,+\infty]$ and $q_2\in(2,+\infty]$.\\
2. In our discussion, the radial assumption of $u$ is necessary. For any $u\in H^1{\mathbb{R}^3}$, we do not know how to deal with the nonlinear term in section 4.\\
3. In this article, we only give the blowup speed, but this rate is not sharp.
\end{remark}
This paper is organized as follows: in section 2, we construct the approximate blowup profile; in section 3, we estimate the parameter is small enough; in section 4, we define the new energy functional and obtain the bound estimate; in section 5, we apply the energy estimate to establish a bootstrap argument; in section 6 we prove the existence of minimal mass blowup solutions; the finally section is Appendix.\\
\textbf{Notations and definitions}\\
- $(f,g)=\int \bar{f}g$ as the inner product on $L^2(\mathbb{R}^3)$.\\
- $\|\cdot\|_{L^p}$ denotes the $L^p(\mathbb{R}^3)$ norm for $p\geq 1$.\\
- $\widehat{f}$ denotes the Fourier transform of function $f$.\\
- We shall use $X\lesssim Y$ to denote that $X\leq CY$ holds, where the constant $C>0$ may change from line to line, but $C$ is allowed to depend on universally fixed quantities only.\\
- Likewise, we use $X\sim Y$ to denote that both $X\lesssim Y$ and $Y\lesssim X$ hold.

For a sufficiently regular function $f:\mathbb{R}^3\rightarrow\mathbb{C}$, we define the generator of $L^2$ scaling given by
\begin{align}\notag
\Lambda f:=\frac{3}{2}f+x\cdot\nabla f.
\end{align}
Note that the operator $\Lambda$ is skew-adjoint on $L^2(\mathbb{R}^3)$, that is, we have
\begin{align}\notag
(\Lambda f, g)=-(f,\Lambda g).
\end{align}
We write $\Lambda^kf$, with $k\in\mathbb{N}$, for the iterates of $\Lambda$ with the convention that $\Lambda^0f\equiv f$.

In some parts of this paper, it will be convenient to identity any complex-valued function $f:\mathbb{R}^3\rightarrow\mathbb{C}$ with the function $\mathbf{f}:\mathbb{R}^3\rightarrow\mathbb{R}^2$ by setting
\begin{align}\notag
\mathbf{f}={\left[ \begin{array}{c}
f_1\\
f_2
\end{array}
\right ]}={\left[ \begin{array}{c}
\Re f\\
\Im f
\end{array}
\right ]}.
\end{align}
Corresponding, we will identity the multiplication by $i$ in $\mathbb{C}$ with the multiplication by the real $2\times 2$-matrix defined as
\begin{align}\notag
J={\left[\begin{array}{cc}
0 & -1\\
1 & 0
\end{array}\right]}.
\end{align}
\section{Approximate Blowup Profile}\label{section2-3D}
This section is devoted to the construction of the approximate blowup profile.
We start with a general observation: If $u=u(t,x)$ solves \eqref{equ-1-hf-3D}, then we define the function $v=v(s,y)$ by setting
\begin{align}\notag
u(t,x)=\frac{1}{\lambda^{\frac{3}{2}}(t)}v\left(t,\frac{x-\alpha(t)}{\lambda(t)}\right)
e^{i\gamma(t)},\ \ \frac{ds}{dt}=\frac{1}{\lambda(t)}.
\end{align}
It is easy to check that $v=v(s,y)$ with $y=\frac{x-\alpha}{\lambda}$ satisfies
\begin{equation}\notag
i\partial_sv-Dv-v+|v|^{\frac{2}{3}}v=i\frac{\lambda_s}{\lambda}\Lambda v+i\frac{\alpha_s}{\lambda}\cdot\nabla v+\tilde{\gamma_s}v,
\end{equation}
where we set $\tilde{\gamma_s}=\gamma_s-1$. Here the operators $D$ and $\nabla$ are understood as $D=D_y$ and $\nabla=\nabla_y$, respectively.  Following the slow modulated ansatz strategy developed in \cite{Raphael-2009-cpam,Raphael2011-Jams,KLR2013}. We freeze the modulation
\begin{align}\notag
-\frac{\lambda_s}{\lambda}=b,\ \ \frac{\alpha_s}{\lambda}=\beta.
\end{align}
Here $\beta$ is a vector. And we look for the approximate solution of the form
\begin{align}\notag
v(s,y)=Q_{\mathcal{P}(s)}(y),\,\,\mathcal{P}(s)=(b(s),\beta(s)),
\end{align}
with an expansion
\begin{align}\notag
Q_{\mathcal{P}}=Q(y)+\sum_{k\geq1}b^{k}R_{k,0}(y)+\sum_{k+l\geq1}b^{k}\sum_{j=1}^3\beta^{l}_j R_{k,l,j},
\end{align}
where $\mathcal{P}=(b,\beta)\in\mathbb{R}\times\mathbb{R}^3$. The terms $R_{k,0}(y)$, $R_{k,l,j}(y)$ are decomposed in real and imaginary  parts as follows
\begin{align}\notag
R_{k,0}(y)=T_{k,0}+iS_{k,0}\,\,\,\text{and}\,\,\,
R_{k,l,j}(y)=T_{k,l,j}+iS_{k,l,j}.
\end{align}
We also use the notation
\begin{align}\notag
T_{k,l}=(T_{k,l,1},T_{k,l,2},T_{k,l,3})\,\,\text{and}\,\,S_{k,l}=(S_{k,l,1},S_{k,l,2},S_{k,l,3}).
\end{align}
We will define ODE for $b(s)$ and $\beta(s)$ of type
\begin{align}\notag
b_s=P_1(b,\beta),\,\,\,\beta_s=P_2(b,\beta),
\end{align}
where $P_1$ and $P_2$ are approximate polynomials in $b$ and $\beta$.
We adjust the modulation equation for $(b_s,\beta_s)$ to ensure the solvability of the obtained  system, and a specific algebra leads to the laws to leading order:
\begin{align}\notag
b_s=-\frac{1}{2}b^2,\ \beta_s=-b\beta.
\end{align}
This allows us to construct a high order approximation $Q_{\mathcal{P}}$ solution to
\begin{align}\notag
-i\frac{b^2}{2}\partial_bQ_{\mathcal{P}}-ib\sum_{j=1}^{3}\beta_j\partial_{\beta_j}Q_{\mathcal{P}}-DQ_{\mathcal{P}}
-Q_{\mathcal{P}}+ib\Lambda Q_{\mathcal{P}}-i\beta\cdot\nabla Q_{\mathcal{P}} +|Q_{\mathcal{P}}|^{\frac{2}{3}}Q_{\mathcal{P}}=-\Phi_{\mathcal{P}},
\end{align}
where $\Phi_{\mathcal{P}}$ is some small and well-localized error term.

We have the following result about an approximate blowup profile $\mathbf{Q}_{\mathcal{P}}$, parameterized by $\mathcal{P}=(b,\beta)$, around the ground state $\mathbf{Q}=[Q,0]^{\top}$.
\begin{lemma}(\textbf{Approximate Blowup Profile})\label{lemma-3app-3D}
Let $\mathcal{P}=(b,\beta)\in\mathbb{R}\times\mathbb{R}^3$. There exists a smooth function $\mathbf{Q}_{\mathcal{P}}=\mathbf{Q}_{\mathcal{P}}(x)$ of the form
\begin{align}\notag
\mathbf{Q}_{\mathcal{P}}=&\mathbf{Q}+b\mathbf{R}_{1,0}+\beta\cdot\mathbf{R}_{0,1}+a\sum_{j=1}^3b_j\mathbf{R}_{1,1,j}
+b^2\mathbf{R}_{2,0}\notag\\&+\sum_{j=1}^3\beta_j^2\mathbf{R}_{0,2,j}+b^3\mathbf{R}_{3,0}+b^2\beta\cdot\mathbf{R}_{2,1}+b^4\mathbf{R}_{4,0}
\end{align}
that satisfies the equation
\begin{align}\label{equ-3approxiamte-3D}
-J\frac{1}{2}b^2\partial_b\mathbf{Q}_{\mathcal{P}}-&Jb\sum_{j=1}^{3}\beta_j\partial_{\beta_j}\mathbf{Q}_{\mathcal{P}}
-D\mathbf{Q}_{\mathcal{P}}-\mathbf{Q}_{\mathcal{P}}+Jb\Lambda\mathbf{Q}_{\mathcal{P}}\notag\\
&-J\sum_{j=1}^{3}\beta_j\partial_{\beta_j}\mathbf{Q}_{\mathcal{P}}
+|\mathbf{Q}_{\mathcal{P}}|^{\frac{2}{3}}\mathbf{Q}_{\mathcal{P}}
=-\mathbf{\Phi}_{\mathcal{P}}.
\end{align}
Here the functions $\{\mathbf{R}_{k,l}\}_{0\leq k\leq3,0\leq l\leq1}$ satisfy the following regularity and decay bounds:
\begin{align}\label{3app-regulairty-3D}
\|\mathbf{R}_{k,l}\|_{H^m}+\|\Lambda\mathbf{R}_{k,l}\|_{H^m}
+\|\Lambda^2\mathbf{R}_{k,l}\|_{H^m}\leq1,\ &\text{for}\ m\in\mathbb{N},\\\label{3app-decay-3D}
|\mathbf{R}_{k,l}|+|\Lambda\mathbf{R}_{k,l}|+|\Lambda^2\mathbf{R}_{k,l}|\leq
\langle x\rangle^{-4},\ &\text{for}\ x\in\mathbb{R}^3.
\end{align}
Moreover, the term on the right-hand side of \eqref{equ-3approxiamte-3D} satisfies
\begin{align}\label{3app-regu-decay-3D}
\|\mathbf{\Phi}_{\mathcal{P}}\|_{H^m}\leq\mathcal{O}(b^5+\beta^2\mathcal{P}),\ |\nabla\mathbf{\Phi}_{\mathcal{P}}|\leq\mathcal{O}(b^5+\beta^2\mathcal{P})\langle x\rangle^{-4},
\end{align}
for $m\in\mathbb{N}$ and $x\in\mathbb{R}^3$.

In addition, the mass, the energy and the linear momentum of $\mathbf{Q}_{\mathcal{P}}$ satisfy
\begin{align*}
\int|\mathbf{Q}_{\mathcal{P}}|^2&=\int  Q^2+\mathcal{O}(b^4+\beta^2+\beta\mathcal{P}^2),\\
E(\mathbf{Q}_{\mathcal{P}})&=e_1b^2+\mathcal{O}(b^4+\beta^2+\beta\mathcal{P}^2),\\
P(\mathbf{Q}_{\mathcal{P}})&=p_1\beta+\mathcal{O}(b^4+\beta^2+\beta\mathcal{P}^2).
\end{align*}
Here $\beta^2=\sum_{j,k=1}^3\beta_j\beta_k$ and $e_1>0$ and $p_1>0$ are the positive constants given by
\begin{align}\notag
e_1=\frac{1}{2}(L_{-}S_{1,0},S_{1,0}),\ \ p_1=2\int_{\mathbb{R}^3}L_{-}S_{0,1}\cdot S_{0,1},
\end{align}
where $S_{1,0}$ and $S_{0,1}$ satisfy $L_{-}S_{1,0}=\Lambda Q$ and $L_{-}S_{0,1}=-\nabla Q$, respectively.
\end{lemma}
\begin{proof}
We recall that the definition of linear operator
\begin{align}\notag
L_+=D+1-\frac{5}{3}Q^{2/3},\ L_-=D+1-Q^{2/3}.
\end{align}
From \cite{FrankLS2016} we have the key property that the kernel of $L_+$ and $L_-$ is given by
\begin{align}\notag
\ker L_+=\{\partial_{x_1}Q,\partial_{x_3}Q,\partial_{x_3}Q\}\,\,\text{and}\,\,
\ker L_-=\{Q\}.
\end{align}
It follows from the above properties (see \cite{Abramowita-1964} for the properties of Helmholtz kernel and \cite[Appendix A]{Cote-Coz-2011} or proof of \cite[Lemma 3.2]{Merle-Raphael-Duke2014} for similar arguments) that
\begin{align}\label{approximate-1-hf-3D}
\forall g\in L^2,\,(g,\nabla Q)=0,\,\exists f_+\in L^2,\,L_+f_+=g,\notag\\
\forall g\in L^2,\,(g, Q)=0,\,\exists f_-\in L^2,\,L_-f_-=g.
\end{align}
Using the above property \eqref{approximate-1-hf-3D}, we discuss our ansatz for $Q_{\mathcal{P}}$ to solve \eqref{equ-approximate-hf-3D} order by order. Following the similar  argument, see \cite{KLR2013,Georgiev-Li}, we can prove this lemma, and here we omit the details.
\end{proof}

\begin{remark}\label{remark-3D}
1. We know that $Q_{\mathcal{P}}$ has the following form
\begin{align*}
Q_{\mathcal{P}}=&Q+ibS_{1,0}+i\beta\cdot S_{0,1}+b\beta\cdot T_{1,1}+b^2T_{2,0}+\sum_{j=1}^3\beta_j^2T_{0,2,j}\\
&+b^3T_{3,0}+ib^2\beta\cdot S_{2,1}+b^4 T_{4,0}.
\end{align*}
2. We have the following identity:
\begin{align}\label{3app-relation-3D}
(S_{1,0},S_{1,0})=-2(T_{2,0},Q).
\end{align}
3. $T_{k,l}$ and $S_{k,l}$ have the following symmetry properties:
\begin{align*}
S_{1,0},T_{2,0}, T_{0,2},S_{3,0},S_{4,0}\,\,\text{are radial symmetry},\\
S_{0,1},T_{1,1},S_{2,1}\,\,\text{are antisymmetry}.
\end{align*}
\end{remark}

\section{Geometrical Decomposition and Modulation Equation}\label{section-mod-estimate-3D}
Let $u\in H^{1/2}(\mathbb{R}^3)$ be a radial solution of \eqref{equ-1-hf-3D} on some time interval $[t_0,t_1]$ with $t_1<0$. Assume that $u(t)$ admits a geometrical decomposition of the form
\begin{align}\label{mod-decomposition-3D}
u(t,x)=\frac{1}{\lambda^{\frac{3}{2}}(t)}[Q_{\mathcal{P}(t)}+\epsilon]
\left(s,\frac{x-\alpha(t)}{\lambda(t)}\right)
e^{i\gamma(t)},\ \ \frac{ds}{dt}=\frac{1}{\lambda(t)},
\end{align}
with $\mathcal{P}(t)=(b(t),\beta(t))$, and we impose the uniform smallness bound
\begin{align}\notag
b^2(t)+|\beta(t)|+\|\epsilon\|_{H^{1/2}}^2\ll1.
\end{align}
Furthermore, we assume that $u(t)$ has almost critical mass in the sense that
\begin{align}\label{mod-mass-assume-3D}
\left|\int|u(t)|^2-\int Q^2\right|\lesssim\lambda^2(t),\ \ \forall t\in[t_0,t_1].
\end{align}
To fix the modulation parameters $\{b(t),\beta(t),\lambda(t),\alpha(t),\gamma(t)\}$ uniquely, we impose the following orthogonality conditions on $\epsilon=\epsilon_1+i\epsilon_2$ as follows:
\begin{align}\label{mod-orthogonality-condition-3D}
(\epsilon_2,\Lambda Q_{1\mathcal{P}})-(\epsilon_1,\Lambda Q_{2\mathcal{P}})=0,\notag\\
(\epsilon_2,\partial_bQ_{1\mathcal{P}})-(\epsilon_1,\partial_bQ_{2\mathcal{P}})=0,\notag\\
(\epsilon_2,\partial_{\beta_j}Q_{1\mathcal{P}})-(\epsilon_1,\partial_{\beta_j}Q_{2\mathcal{P}})=0,\notag\\
(\epsilon_2,\partial_{x_j} Q_{1\mathcal{P}})-(\epsilon_1,\partial_{x_j} Q_{2\mathcal{P}})=0,\notag\\
(\epsilon_2,\rho_1)+(\epsilon_1,\rho_2)=0,
\end{align}
 the function $\rho=\rho_1+i\rho_2$ is defined by
\begin{align}\label{mod-definition-rho-3D}
&L_{+}\rho_1=S_{1,0},\notag\\ &L_{-}\rho_2=b\frac{2}{3}Q^{-\frac{1}{3}}S_{1,0}\rho_1+b\Lambda\rho_1-2bT_{2,0}
+\frac{2}{3}\beta\cdot Q^{-\frac{1}{3}}S_{0,1}\rho_1-\beta\cdot\nabla\rho_1-\beta\cdot T_{1,1},
\end{align}
where $S_{1,0}$, $T_{2,0}$ and $T_{1,1}$ are the functions, see Remark \ref{remark-3D}. Note that $L_{+}^{-1}$ exists on $L^2_{rad}(\mathbb{R}^3)$ and thus $\rho_1$ is well-defined. Moreover, it is easy to see that the right-hand side in the equation for $\rho_2$ is orthogonality to $Q$. Indeed
\begin{align*}
(Q,\frac{2}{3}Q^{-\frac{1}{3}}S_{1,0}\rho_1+\Lambda\rho_1-2T_{2,0})&=
\frac{2}{3}(Q^{\frac{2}{3}}S_{1,0},\rho_1)-(\Lambda Q,\rho_1)-2(Q,T_{2,0})\\
&=\frac{2}{3}(Q^{\frac{2}{3}}S_{1,0},\rho_1)-(S_{1,0},L_{-}\rho_1)+(S_{1,0},S_{1,0})\\
&=-(S_{1,0},L_{+}\rho_1)+(S_{1,0},S_{1,0})=0,
\end{align*}
using that $(S_{1,0},S_{1,0})=-2(T_{2,0},Q)$, see \eqref{3app-relation-3D}, and the definition of $\rho_1$. Moreover, we clearly see that $\frac{2}{3}Q^{-\frac{1}{3}}\beta\cdot S_{0,1}\rho_1-\beta\cdot\nabla\rho_1-\beta\cdot T_{1,1}\bot Q$, since $S_{0,1,j}$ and $T_{1,1,j}$ are the antisymmetry functions, whereas $\rho_1$ and $Q$ are radial symmetry functions. Hence $\rho_2$ is well-defined.

In the conditions \eqref{mod-orthogonality-condition-3D}, we use the notation
\begin{align}\notag
Q_{\mathcal{P}}=Q_{1\mathcal{P}}+iQ_{2\mathcal{P}},
\end{align}
which (in terms of the vector notation used in Section 3) means that
\begin{align}\notag
\mathbf{Q}_{\mathcal{P}}={\left[\begin{array}{c}
Q_{1\mathcal{P}}\\Q_{2\mathcal{P}}
\end{array}
\right]}.
\end{align}

By the similar arguments as \cite{KLR2013,Georgiev-Li,MerleR2006}, we can obtain  the modulation parameters $$\{b(t),~\beta(t),~\lambda(t),~\alpha(t),~\gamma(t)\}$$ are uniquely determined, provided that $\epsilon=\epsilon_1+i\epsilon_2\in H^{1/2}(\mathbb{R}^3)$ is sufficiently small. Moreover, it follows from the standard arguments that $\{b(t),\beta(t),\lambda(t),\alpha(t),\gamma(t)\}$ are $C^1$-functions.

If we insert the decomposition \eqref{mod-decomposition-3D} into \eqref{equ-1-hf-3D}, we obtain the following system
\begin{align}\label{mod-system-1-3D}
&(b_s+\frac{1}{2}b^2)\partial_bQ_{1\mathcal{P}}+(\beta_s+b\beta)\cdot\partial_{\beta}Q_{1\mathcal{P}}
+\partial_s\epsilon_1-M_{-}(\epsilon)+b\Lambda\epsilon_1-\beta\cdot\nabla\epsilon_1\notag\\
=&(\frac{\lambda_s}{\lambda}+b)(\Lambda Q_{1\mathcal{P}}+\Lambda\epsilon_1)
+(\frac{\alpha_s}{\lambda}-\beta)\cdot(\nabla Q_{1\mathcal{P}}+\nabla\epsilon_1)\notag\\
&+\tilde{\gamma}_s(Q_{2\mathcal{P}}+\epsilon_2)+\Im(\Phi_{\mathcal{P}})
-R_2(\epsilon),\\\label{mod-system-2-3D}
&(b_s+\frac{1}{2}b^2)\partial_bQ_{2\mathcal{P}}+(\beta_s+b\beta)\cdot\partial_{\beta}Q_{2\mathcal{P}}
+\partial_s\epsilon_2+M_{+}(\epsilon)+b\Lambda\epsilon_2-\beta\cdot\nabla\epsilon_2\notag\\
=&(\frac{\lambda_s}{\lambda}+b)(\Lambda Q_{2\mathcal{P}}+\Lambda\epsilon_2)
+(\frac{\alpha_s}{\lambda}-\beta)\cdot(\nabla Q_{2\mathcal{P}}+\nabla\epsilon_2)\notag\\
&-\tilde{\gamma}_s(Q_{1\mathcal{P}}-\epsilon_1)-\Re(\Phi_{\mathcal{P}})+R_1(\epsilon).
\end{align}
Here $\Phi_{\mathcal{P}}$ denotes the error term from lemma \ref{lemma-3app-3D}, and $M=(M_{+},M_{-})$ are the small deformations of the linearized operator $L=(L_{+},L_{-})$ given by
\begin{align}\label{mod-define-M1-3D}
M_{+}(\epsilon)=&D\epsilon_1+\epsilon_1
-\frac{4}{3}|Q_{\mathcal{P}}|^{\frac{2}{3}}\epsilon_1
-\frac{1}{3}|Q_{\mathcal{P}}|^{-\frac{4}{3}}(Q_{1\mathcal{P}}^2-Q_{2\mathcal{P}}^2)\epsilon_1\notag\\
&-\frac{2}{3}|Q_{\mathcal{P}}|^{-\frac{4}{3}}Q_{1\mathcal{P}}Q_{2\mathcal{P}}\epsilon_2,\\\label{mod-define-M2}
M_{-}(\epsilon)=&D\epsilon_2+\epsilon_2
-\frac{4}{3}|Q_{\mathcal{P}}|^{\frac{2}{3}}\epsilon_2
-\frac{1}{3}|Q_{\mathcal{P}}|^{-\frac{4}{3}}(Q_{1\mathcal{P}}^2-Q_{2\mathcal{P}}^2)\epsilon_2\notag\\
&-\frac{2}{3}|Q_{\mathcal{P}}|^{-\frac{4}{3}}Q_{1\mathcal{P}}Q_{2\mathcal{P}}\epsilon_1.
\end{align}
And $R_1(\epsilon)$, $R_2(\epsilon)$ are the high order terms about $\epsilon$.
\begin{align*}
R_1(\epsilon)&=\frac{2}{3}|Q_{\mathcal{P}}|^{-\frac{4}{3}}(Q_{1\mathcal{P}}\epsilon_1
+Q_{2\mathcal{P}}\epsilon_2)\epsilon_1
+\frac{1}{3}|Q_{\mathcal{P}}|^{-\frac{4}{3}}|\epsilon|^2Q_{1\mathcal{P}}\\
&-\frac{4}{9}|Q_{\mathcal{P}}|^{-\frac{10}{3}}(Q_{1\mathcal{P}}\epsilon_1
+Q_{2\mathcal{P}}\epsilon_2)Q_{1\mathcal{P}}+\mathcal{O}(\epsilon^3)\\
R_2(\epsilon)&=\frac{2}{3}|Q_{\mathcal{P}}|^{-\frac{4}{3}}(Q_{1\mathcal{P}}\epsilon_1
+Q_{2\mathcal{P}}\epsilon_2)\epsilon_2
+\frac{1}{3}|Q_{\mathcal{P}}|^{-\frac{4}{3}}|\epsilon|^2Q_{2\mathcal{P}}\\
&-\frac{4}{9}|Q_{\mathcal{P}}|^{-\frac{10}{3}}(Q_{1\mathcal{P}}\epsilon_1
+Q_{2\mathcal{P}}\epsilon_2)Q_{2\mathcal{P}}+\mathcal{O}(\epsilon^3).
\end{align*}

We have the following energy type bound.
\begin{lemma}\label{lemma-mod-1-3D}
For $t\in[t_0,t_1]$, it holds that
\begin{align}\notag
b^2+|\beta|+\|\epsilon\|_{H^{1/2}}^2\lesssim\lambda(|E_0|+|P_0|)
+\mathcal{O}(\lambda^2+b^4+\beta^2+\beta\mathcal{P}^2).
\end{align}
Here $E_0=E(u_0)$ and $P_0=P(u_0)$ denote the conserved energy and linear momentum of $u=u(t,x)$, respectively.
\end{lemma}
\begin{proof}
By the similar arguments as \cite{Georgiev-Li}, we can obtain this estimate. Here we omit the details.
\end{proof}

We continue with estimating the modulation parameters. To this end, we define the vector-valued function
\begin{align}\label{define-mod-3D}
\mathbf{Mod}(t):=\left(b_s+\frac{1}{2}b^2,\tilde{\gamma}_s,\frac{\lambda_s}{\lambda}+b,
\frac{\alpha_s}{\lambda}-\beta,\beta_s+b\beta\right).
\end{align}
We have the following result.
\begin{lemma}\label{lemma-mod-2-3D}
For $t\in[t_0,t_1]$, we have the bound
\begin{align}
|\mathbf{Mod}(t)|\leq\lambda^2+b^4+\beta^2+\beta\mathcal{P}^2+\mathcal{P}^2\|\epsilon\|_2+
\|\epsilon\|_2^2+\|\epsilon\|_{H^{1/2}}^3.
\end{align}
Furthermore, we have the improved bound
\begin{align}\notag
\left|\frac{\lambda_s}{\lambda}+b\right|\leq b^5 +\beta\mathcal{P}^2+\mathcal{P}^2\|\epsilon\|_2
+\|\epsilon\|_2^2+\|\epsilon\|_{H^{1/2}}^3.
\end{align}
\end{lemma}
\begin{proof} We shall give only the sketch of the proof since the details can be found in Lemma 4.2 in \cite{Georgiev-Li}.

$\mathbf{Law~ for}$ $b$. Here the treatment is quite close to the corresponding proof of Lemma 4.2 in \cite{Georgiev-Li} so we can write
\begin{align*}
&-\left(b_s+\frac{1}{2}b^2\right)[2e_1+\mathcal{O}(\mathcal{P}^2)]
+\sum_{j=1}^3\left(\frac{(\alpha_j)_s}{\lambda}-\beta_j\right)\left[\frac{1}{2}{p_j}_1+\mathcal{O}(\mathcal{P}^2)\right]\\
=&-\int|\epsilon|^2+(R_2(\epsilon),\Lambda Q_{2\mathcal{P}})+(R_1(\epsilon),\Lambda Q_{1\mathcal{P}})\\
&+\mathcal{O}\left((\mathcal{P}^2+|\mathbf{mod}(t)|)
(\|\epsilon\|_2+\mathcal{P}^2)+|\|u\|_2^2-\|Q\|_2^2|+b^4+\beta^2+\beta\mathcal{P}^2\right).
\end{align*}
By the similar arguments, we can obtain the following estimates:

$\mathbf{Law~ for}$ $\lambda$.
\begin{align*}
&(\frac{\lambda_s}{\lambda}+b)[2e_1+\mathcal{O}(\mathcal{P}^2)]
+\sum_{j=1}^3({\beta_j}_s+b\beta_j)\mathcal{O}(\mathcal{P})\\
=&(R_2(\epsilon),\partial_b Q_{2\mathcal{P}})+(R_1(\epsilon),\partial_b Q_{1\mathcal{P}})+\mathcal{O}\left((\mathcal{P}^2
+|\mathbf{mod}(t)|)(\|\epsilon\|_2+\mathcal{P}^2)+b^5+\beta^2\mathcal{P}\right).
\end{align*}
Furthermore, by this estimate, we deduce the improved bound for $\left|\frac{\lambda_s}{\lambda}+b\right|$.

$\mathbf{Law~ for}$ $\tilde{\gamma}_s$.
\begin{align*}
&\tilde{\gamma}_s((Q,\rho_1)+\mathcal{O}(\mathcal{P}^2))\\
=&-\left(b_s+\frac{1}{2}b^2\right)
\big((S_{1,0},\rho_1)+\mathcal{O}(\mathcal{P}^2)\big)
+\left(\frac{\lambda_s}{\lambda}+b\right)\mathcal{O}(\mathcal{P})
+\sum_{j=1}^3\left(\frac{(\alpha_j)_s}{\lambda}-\beta_j\right)\mathcal{O}(\mathcal{P})\\
&+(R_2(\epsilon),\rho_2)+(R_1(\epsilon),\rho_1)+\mathcal{O}\left((\mathcal{P}^2
+|\mathbf{mod}(t)|)(\|\epsilon\|_2+\mathcal{P}^2)+b^5+\beta^2\mathcal{P}\right).
\end{align*}

$\mathbf{Law~ for}$ $\beta_j,\,j=1,2,3$.
\begin{align*}
&({\beta_j}_s+b\beta_j)\left[-\frac{1}{2}p_{1,j}+\mathcal{O}(\mathcal{P}^2)\right]\\
=&(R_2(\epsilon),\partial_{x_j} Q_{2\mathcal{P}})+(R_1(\epsilon),\partial_{x_j} Q_{1\mathcal{P}})
+\mathcal{O}\left((\mathcal{P}^2
+|\mathbf{mod}(t)|)\|\epsilon\|_2+b^4+\beta^2+\beta\mathcal{P}^2\right).
\end{align*}

$\mathbf{Law~ for}$ $\alpha_j,~j=1,2,3$.
\begin{align*}
&\left(b_s+\frac{1}{2}b^2\right)\mathcal{O}(\mathcal{P})
+\left(\frac{(\alpha_j)_s}{\lambda}-\beta_j\right)[p_{1,j} +\mathcal{O}(\mathcal{P}^2)]\\
=&(R_2(\epsilon),\partial_{\beta_j} Q_{2\mathcal{P}})+(R_1(\epsilon),\partial_{\beta_j} Q_{1\mathcal{P}})
+\mathcal{O}\left((\mathcal{P}^2
+|\mathbf{mod}(t)|)\|\epsilon\|_2+b^4+\beta^2+\beta\mathcal{P}^2\right).
\end{align*}

$\mathbf{Conclusion.}$  We collect the previous equation and estimate the nonlinear terms in $\epsilon$ by Sobolev inequalities. This gives us
\begin{align*}
(A+B)\mathbf{Mod}(t)=&\mathcal{O}\big((\mathcal{P}^2+|\mathbf{Mod}(t)|)\|\epsilon_2\|
+\|\epsilon\|_2^2+\|\epsilon\|_{H^{1/2}}^3\\
&+|\|u\|_2^2-\|Q\|_2^2|+b^4+\beta^2+\beta\mathcal{P}^2\big).
\end{align*}
Here $A=O(1)$ is invertible $9\times9$-matrix, and $B=\mathcal{O}(\mathcal{P})$ is some $9\times9$-matrix that is polynomial in $\mathcal{P}=(b,\beta)$. For $|\mathcal{P}|\ll1$, we can thus invert $A+B$ by Taylor expansion and derive the estimate for $\mathbf{Mod}(t)$ stated in this lemma.
\end{proof}

\section{Refined Energy bounds}\label{section-refined-energy-3D}
In this section, we establish a refined energy estimate, which will be a key ingredient in the compactness argument to construct minimal mass blowup solutions. Let $u=u(t,x)$ be a radial solution \eqref{equ-1-hf-3D} on the time interval $[t_0,0)$ and suppose that $w$ is a radial approximate solution to \eqref{equ-1-hf-3D} such that
\begin{equation}\label{equ-approximate-hf-3D}
iw_t-Dw+|w|^{\frac{2}{3}}w=\psi,
\end{equation}
with the priori bounds
\begin{align}\label{energy-priori-1-3D}
\|w\|_2\lesssim 1,\ \|D^{\frac{1}{2}}w\|_2\lesssim \lambda^{-\frac{1}{8}},\ \|\nabla w\|_2\lesssim \lambda^{-\frac{1}{4}}.
\end{align}
We decompose $u=w+\tilde{u}$, and hence $\tilde{u}$ is radial and satisfies
\begin{equation}\label{equ-app2-hf-3D}
i\tilde{u}_t-D\tilde{u}+(|u|^{\frac{2}{3}}u-|w|^{\frac{2}{3}}w)=-\psi,
\end{equation}
where we assume the priori estimate
\begin{align}\label{energy-priori-2-3D}
\|D^{\frac{1}{2}}\tilde{u}\|_2\lesssim \lambda^{\frac{1}{2}},\ \|\tilde{u}\|_2\lesssim \lambda,
\end{align}
as well as
\begin{align}\label{energy-priori-3-3D}
|\lambda_t+b|\lesssim\lambda^2,\ b\lesssim\lambda^{\frac{1}{2}},\ |b_t|\lesssim1,\ |\alpha_t|\lesssim\lambda.
\end{align}
Next, Let $\phi:\mathbb{R}\rightarrow\mathbb{R}$ be a smooth and radial function with the following properties
\begin{align}\label{energy-cutoff-function-3D}
\phi'(x)=\begin{cases}x\ \ &\text{for}\ \ 0\leq x\leq1,\\
3-e^{-|x|}\ &\text{for}\ x\geq2,
\end{cases}
\end{align}
and the convexity condition
\begin{align}
\phi''(x)\geq0\ \ \text{for}\ \ x\geq0.
\end{align}
Furthermore, we denote
\begin{align}\notag
F(u)=\frac{3}{8}|u|^{\frac{8}{3}},\ f(u)=|u|^{\frac{2}{3}}u,\ F'(u)\cdot h=\Re(f(u)\bar{h}).
\end{align}

Let $A>0$ be a large constant and define the quantity
\begin{align}\notag
J_A(u):=&\frac{1}{2}\int|D^{\frac{1}{2}}\tilde{u}|^2
+\frac{1}{2}\int\frac{|\tilde{u}|^2}{\lambda}
-\int[F(u)-F(w)-F'(w)\cdot\tilde{u}]\notag\\
&+\frac{b}{2}\Im\left(\int A\nabla\phi\left(\frac{x-\alpha}{A\lambda}\right)\cdot\nabla\tilde{u}\bar{\tilde{u}}\right).
\end{align}
Our strategy will be to use the preceding functional to bootstrap control over $\|\tilde{u}\|_{H^{\frac{1}{2}}}$.
\begin{lemma}\label{lemma-energy-estimate-3D}
(Localized energy estimate) Let $J_A$ be as above. Then we have
\begin{align}
\frac{dJ_A}{dt}
=&-\Im\left(\psi,D\tilde{u}+\frac{1}{\lambda}\tilde{u}-f'(w)\tilde{u}\right)
-\frac{1}{\lambda}(\tilde{u},f'(w)\tilde{u})\notag\\
&+\frac{b}{2\lambda}\int\frac{|\tilde{u}|^2}{\lambda}
-\frac{2b}{\lambda}\int_0^{+\infty}\sqrt{s}
\int_{\mathbb{R}^3}\Delta\phi(\frac{x-\alpha}{A\lambda})|\nabla\tilde{u}_s|^2dxds\notag\\
&+\frac{b}{2A^2\lambda^3}\int_0^{+\infty}\sqrt{s}\int_{\mathbb{R}^3}
\Delta^2\phi(\frac{x-\alpha}{A\lambda})|\tilde{u}_s|^2dxds\notag\\
&+\Im\left(\int\left[ibA\nabla\phi(\frac{x-\alpha}{A\lambda})\cdot\nabla\psi
+i\frac{b}{2\lambda}\Delta\psi(\frac{x-\alpha}{A\lambda})\psi\right]\bar{\tilde{u}}\right)\notag\\
&-b\Re\left(\int A\nabla\phi(\frac{x-\alpha}{A\lambda})
(\frac{4}{3}|w|^{\frac{2}{3}}\tilde{u}
+\frac{1}{3}|w|^{-\frac{4}{3}}w^2\bar{\tilde{u}})\cdot\overline{\nabla \tilde{u}}\right)\notag\\
&-\frac{1}{2}\frac{b}{\lambda}\Re\left(\int\Delta\phi(\frac{x-\alpha}{A\lambda})
(\frac{4}{3}|w|^{\frac{2}{3}}\tilde{u}
+\frac{1}{3}|w|^{-\frac{4}{3}}w^2\bar{\tilde{u}})\cdot\overline{\tilde{u}}\right)\notag\\
&+\mathcal{O}\left(\|\tilde{u}\|_{H^{1/2}}^2
+\|\tilde{u}\|_{H^{1/2}}^\frac{5}{3}
+\|\tilde{u}\|_{H^{1/2}(\mathbb{R}^3)}^{\frac{7}{6}}
+\lambda^{\frac{5}{6}}\|\psi\|_2
+H(t)^{1/2}\right),
\end{align}
where $\tilde{u}_s:=\sqrt{\frac{2}{\pi}}\frac{1}{-\Delta+s}\tilde{u}$ with $s>0$.
\end{lemma}
\begin{proof}
$\mathbf{Step~1:}$ (Estimating the energy part).
Using \eqref{equ-app2-hf-3D}, a computation
\begin{align}\label{energy-part-1-3D}
&\frac{d}{dt}\left\{\frac{1}{2}\int|D^{\frac{1}{2}}\tilde{u}|^2
+\frac{1}{2}\int\frac{|\tilde{u}|^2}{\lambda}
-\int[F(u)-F(w)-F'(w)\cdot\tilde{u}]\right\}\notag\\
=&\Re\left(\partial_t\tilde{u},
{D\tilde{u}+\frac{1}{\lambda}\tilde{u}-(f(u)-f(w))}\right)
-\frac{\lambda_t}{2\lambda^2}\int|\tilde{u}|^2\notag\\
&-\Re\left(\partial_tw,{(f(u)-f(w)-f'(w)\cdot\tilde{u})}\right)\notag\\
=&-\Im\left(\psi,{D\tilde{u}+\frac{1}{\lambda}\tilde{u}-(f(u)-f(w))}\right)
-\frac{\lambda_t}{2\lambda^2}\int|\tilde{u}|^2
-\Re\left(\partial_tw,{(f(u)-f(w)-f'(w)\cdot\tilde{u})}\right)\notag\\
&-\Im\left(D\tilde{u}-(f(u)-f(w)),
{D\tilde{u}+\frac{1}{\lambda}\tilde{u}-(f(u)-f(w))}\right)\notag\\
=&-\Im\left(\psi,{D\tilde{u}+\frac{1}{\lambda}\tilde{u}-(f(u)-f(w))}\right)
-\frac{\lambda_t}{2\lambda^2}\int|\tilde{u}|^2
+\Im\left(f(u)-f(w),\frac{1}{\lambda}\tilde{u}\right)\notag\\
&-\Re\left(\partial_tw,{(f(u)-f(w)-f'(w)\cdot\tilde{u})}\right)\notag\\
=&-\Im\left(\psi,D\tilde{u}+\frac{1}{\lambda}\tilde{u}-f'(w)\tilde{u}\right)
-\frac{1}{\lambda}(\tilde{u},f'(w)\tilde{u})-\frac{\lambda_t}{2\lambda^2}\int|\tilde{u}|^2\notag\\
&+\Im\left(\psi-\frac{1}{\lambda}\tilde{u},f(u)-f(w)-f'(w)\cdot\tilde{u}\right)
-\Re\left(\partial_tw,{(f(u)-f(w)-f'(w)\cdot\tilde{u})}\right),
\end{align}
where we denote
\begin{align}\notag
f'(w)\tilde{u}=\frac{4}{3}|w|^{\frac{2}{3}}\tilde{u}
+\frac{1}{3}|w|^{-\frac{4}{3}}w^2\bar{\tilde{u}}.
\end{align}
From \eqref{energy-priori-3-3D} we obtain that
\begin{align}\label{energy-estimate-1-3D}
-\frac{\lambda_t}{2\lambda^2}\int|\tilde{u}|^2
&=\frac{b}{2\lambda}\int\frac{|\tilde{u}|^2}{\lambda}
-\frac{1}{2\lambda^2}(\lambda_t+b)\|\tilde{u}\|_2^2\notag\\
&=\frac{b}{2\lambda}\int\frac{|\tilde{u}|^2}{\lambda}
+\mathcal{O}(\|\tilde{u}\|_{H^{1/2}}^2).
\end{align}
Next, we estimate
\begin{align}\label{energy-estimate-2-3D}
&\left|\Im\left(\psi-\frac{1}{\lambda}\tilde{u},
f(u)-f(w)-f'(w)\cdot\tilde{u}\right)\right|\notag\\
\lesssim&\left(\|\psi\|_2+\frac{1}{\lambda}\|\tilde{u}\|_2\right)
\|f(u)-f(w)-f'(w)\cdot\tilde{u}\|_2\notag\\
\lesssim&\left(\|\psi\|_2+\frac{1}{\lambda}\|\tilde{u}\|_2\right)
\|\tilde{u}^{\frac{2}{3}+1}\|_2\notag\\
\lesssim&\left(\|\psi\|_2+\frac{1}{\lambda}\|\tilde{u}\|_2\right)
\||x|^{\frac{1}{q}}\tilde{u}^{\frac{2}{3}+1}\|_2.
\end{align}
Here in the last step we use the inequality
\begin{align}\notag
\left\||x|^{-\frac{1}{q_1}}\int_0^te^{i(t-s)D}F(s)ds\right\|_{L^{q_1}(\mathbb{R});L^2(\mathbb{R}^N)}
\leq C\left\||x|^{\frac{1}{q_2}}F\right\|_{L^{q_2}(\mathbb{R});L^2(\mathbb{R}^N)},
\end{align}
where $q_1\in[2,+\infty]$ and $q_2\in(2,+\infty]$.
This inequality can be find \cite[Proposition 2.3]{BGV2018} and we can choose $q\in(2,\infty]$.
From Sobolev embedding and the Strauss inequality \cite{BGV2018}, we know that
\begin{align}\notag
\|u\|_{L^{\frac{6}{3-2s_1}}}\leq\|u\|_{H^{s_1}}\,\,\text{and}\,\,\||x|u\|_{L^{\infty}}\leq\|u\|_{H^{\frac{1}{2}+\delta}},\,\,\text{where}\, \delta>0.
\end{align}
Using above two inequalities and Stein-interpolation theorem, we have
\begin{align}\label{energy-estimate-3-3D}
\||x|^{\frac{1}{q}}|u|^{\frac{5}{3}}\|_{L^2}
=\||x|^{\frac{3}{5q}}u\|_{\frac{10}{3}}^{\frac{5}{3}}
\leq\|u\|_{H^s(\mathbb{R}^3)}^{\frac{5}{3}}
\end{align}
where
\begin{align}\notag
\frac{3}{10}=\frac{1-\theta}{\frac{6}{3-2s_1}}+\frac{\theta}{\infty},\,\,\,
\frac{3}{5q}=1\cdot\theta+0\cdot(1-\theta)=\theta\,\,
\text{and}\,\,\,s=\left(\frac{1}{2}+\delta\right)\theta+(1-\theta)s_1.
\end{align}
Hence
\begin{align}
\theta=1-\frac{9}{5(3-2s_1)}\,\,\text{ and}\,\,s=-\frac{2}{5}+\frac{9}{5(3-2s_1)}+\theta\delta.
\end{align}
On the other hand, $q\in(2,\infty]$, from above we know that $q=\frac{3}{5\theta}>2$, hence we have
\begin{align}
s_1>\frac{3}{14}.
\end{align}
Let $\delta>0$ is small enough and $s_1>\frac{3}{14}$ and close to $\frac{3}{14}$, we have $s\leq\frac{8}{25}<\frac{1}{2}$ and close to $\frac{8}{25}$.  Combining \eqref{energy-estimate-2-3D} and \eqref{energy-estimate-3-3D}, we deduce that
\begin{align}\label{energy-estimate-4-3D}
&\left|\Im\left(\psi-\frac{1}{\lambda}\tilde{u},
f(u)-f(w)-f'(w)\cdot\tilde{u}\right)\right|\notag\\
\lesssim&\left(\|\psi\|_2+\frac{1}{\lambda}\|\tilde{u}\|_2\right)
\||x|^{\frac{1}{q}}\tilde{u}^{\frac{2}{3}+1}\|_2\notag\\
\lesssim&\left(\|\psi\|_2+\frac{1}{\lambda}\|\tilde{u}\|_2\right)
\|u\|_{H^s(\mathbb{R}^3)}^{\frac{5}{3}}\notag\\
\lesssim&\left(\|\psi\|_2+\frac{1}{\lambda}\|\tilde{u}\|_2\right)
\|u\|_{H^{1/2}(\mathbb{R}^3)}^{\frac{5}{3}}\notag\\
\lesssim&\lambda^{\frac{5}{6}}\|\psi\|_2+\|\tilde{u}\|_{H^{1/2}}^\frac{5}{3}.
\end{align}
 For the terms that conclude the $\partial_tw$, we use the equation for $w$ and bound \eqref{energy-priori-1-3D}, \eqref{energy-priori-2-3D} and \eqref{energy-estimate-3-3D}. This leads us to
\begin{align}\label{energy-estimate-5-3D}
\left|\int\partial_tw|\tilde{u}|^{\frac{2}{3}+1}\right|
\lesssim&\int (Dw-|w|^{\frac{2}{3}}w+\psi)|\tilde{u}|^{\frac{2}{3}+1}\notag\\
\lesssim&\int Dw|\tilde{u}|^{\frac{5}{3}}
+\|w^{\frac{5}{3}}\|_2\|\tilde{u}^{\frac{5}{3}}\|_2+\|\psi\|_2\|\tilde{u}^{\frac{5}{3}}\|_2\notag\\
\lesssim&\left(\|Dw\|_2+\|Dw\|_2\|w\|_2^{\frac{2}{3}}+\|\psi\|_2\right)
\|\tilde{u}^{\frac{5}{3}}\|_2\notag\\
\lesssim&\|Dw\|_2\|\tilde{u}\|_{H^{8/25}(\mathbb{R}^3)}^{\frac{5}{3}}
+\|\psi\|_2\|\tilde{u}\|_{H^{1/2}(\mathbb{R}^3)}^{\frac{5}{3}}\notag\\
\lesssim&\|\tilde{u}\|_{H^{1/2}(\mathbb{R}^3)}^{\frac{7}{6}}
+\lambda^{\frac{5}{6}}\|\psi\|_2,
\end{align}
where in the penultimate step we have used estimate \eqref{energy-estimate-3-3D}.

Finally, by the Sobolev inequality, H\"{o}lder inequality and Young inequality, we estimate
\begin{align}\label{energy-estimate-6-3D}
|-\frac{1}{\lambda}(\tilde{u},f'(w)\tilde{u})|
\lesssim&\frac{1}{\lambda}\|w^{\frac{2}{3}}\|_{9}\|\tilde{u}^2\|_{9/8}\notag\\
\lesssim&\frac{1}{\lambda}\|\nabla w\|_2^{2/3}
\|D^{\frac{1}{2}}\tilde{u}\|_2^{2/3}\|\tilde{u}\|_2^{4/3}\notag\\
\lesssim&\frac{1}{\lambda^{5/6}}\|\nabla w\|_2^{2/3}\|\tilde{u}\|_2
\left(\frac{2\|D^{\frac{1}{2}}\tilde{u}\|_2}{3}+\frac{\|\tilde{u}\|_2}{3\lambda^{1/2}}\right)\notag\\
\lesssim&\left(\frac{2\|D^{\frac{1}{2}}\tilde{u}\|_2}{3}+\frac{\|\tilde{u}\|_2}{3\lambda^{1/2}}\right)\notag\\
\lesssim& H(t)^{1/2},
\end{align}
where $H(t)$ defines in the next section see \eqref{back-define-norm-1-3D}.

We now insert \eqref{energy-estimate-1-3D}, \eqref{energy-estimate-4-3D}, \eqref{energy-estimate-5-3D} and \eqref{energy-estimate-6-3D} into \eqref{energy-part-1-3D}. Combined with the assumed a priori bounds on $\tilde{u}$, we conclude
\begin{align}
&\frac{d}{dt}\left\{\frac{1}{2}\int|D^{\frac{1}{2}}\tilde{u}|^2
+\frac{1}{2}\int\frac{|\tilde{u}|^2}{\lambda}
-\int[F(u)-F(w)-F'(w)\cdot\tilde{u}]\right\}\notag\\
=&-\Im\left(\psi,D\tilde{u}+\frac{1}{\lambda}\tilde{u}-f'(w)\tilde{u}\right)
-\Re\left(\partial_tw,{(f(u)-f(w)-f'(w)\cdot\tilde{u})}\right)\notag\\
&+\frac{a}{2\lambda}\int\frac{|\tilde{u}|^2}{\lambda}
+\mathcal{O}\left(\|\tilde{u}\|_{H^{1/2}}^2
+\|\tilde{u}\|_{H^{1/2}}^\frac{5}{3}
+\|\tilde{u}\|_{H^{1/2}(\mathbb{R}^3)}^{\frac{7}{6}}
+\lambda^{\frac{5}{6}}\|\psi\|_2
+H(t)^{1/2}\right).
\end{align}

$\mathbf{Step~2:}$ Estimating the localized virial part.
Using \eqref{equ-app2-hf-3D}, we can obtain
\begin{align}\label{energy-prior-part-2-3D}
&\frac{1}{2}\frac{d}{dt}\left(b\Im\left(\int A\nabla
\Big(\frac{x-\alpha}{A\lambda}\Big)\cdot\nabla{\tilde{u}\bar{\tilde{u}}}\right)\right)\notag\\
=&\frac{1}{2}\Im\left(\int(\partial_t\nabla{\tilde{\phi}})\cdot\nabla{u}\bar{\tilde{u}}\right)
+\frac{1}{2}\Im\left(\int\nabla{\tilde{\phi}}\cdot(\nabla\partial_t\tilde{u}\bar{\tilde{u}}
+\nabla\tilde{u}\partial_t\bar{\tilde{u}})\right)\notag\\
=&\frac{1}{2}\Im\left(\int(\partial_t\nabla{\tilde{\phi}})\cdot\nabla{u}\bar{\tilde{u}}\right)
-\frac{1}{4}\Re\left(\int\bar{\tilde{u}}\left[-iD,
\nabla\tilde{\phi}\cdot (-i\nabla)+(-i\nabla)\cdot\nabla\tilde{\phi}\right]\right)\notag\\
&-b\Re(\int(|u|^{\frac{2}{3}}u-|w|^{\frac{2}{3}}w)A\nabla\phi\left(\frac{x-\alpha}{A\lambda}\right)
\cdot\overline{\nabla\tilde{u}})\notag\\
&-\frac{1}{2}\frac{b}{\lambda}\Re\left(\int(|u|^{\frac{2}{3}}u-|w|^{\frac{2}{3}}w)
A\Delta\phi\left(\frac{x-\alpha}{A\lambda}\right)\cdot\bar{\tilde{u}}\right)\notag\\
&-b\Re\left(\int\psi\nabla\phi\left(\frac{x-\alpha}{A\lambda}\right)
\cdot\overline{\nabla\tilde{u}}\right)-\frac{1}{2}\frac{b}{\lambda}\Re
\left(\int\psi\Delta\left(\frac{x-\alpha}{A\lambda}\right)\bar{\tilde{u}}\right),
\end{align}
where $\nabla\tilde{\phi}(t,x)=bA\nabla\phi\left(t,\frac{x-\alpha}{A\lambda}\right)$.
Using the bounds \eqref{energy-priori-3-3D}, we estimate
\begin{align}\notag
|\partial_t\nabla\tilde{\phi}|\leq|b_t|+\frac{b}{\lambda}\alpha_t
+b\frac{\lambda_t}{\lambda}\leq1,\,\text{and}\,\,
|\partial_t\Delta\tilde{\phi}|\leq\lambda^{-1}
\end{align}
Hence, by \cite[Lemma F.1]{KLR2013}, we deduce that
\begin{align}
\left|\Im\left((\partial_t\nabla{\tilde{\phi}})\cdot\nabla{u}\bar{\tilde{u}}\right)\right|
\leq\|\tilde{u}\|_{\dot{H}^{1/2}}^2+\lambda^{-1}\|\tilde{u}\|_2^2.
\end{align}
Recalling that $\nabla\tilde{\phi}(t,x)=bA\nabla\phi(\frac{x-\alpha}{A\lambda})$ and using that $(-\Delta+s)^{-1}$ is self-adjoint and the definition of  $\tilde{u}_s$,  we conclude that
\begin{align}\label{energy-part2-4-3D}
&\Re\left(\int\bar{\tilde{u}}\left[-iD,
\nabla\tilde{\phi}\cdot (-i\nabla)
+(-i\nabla)\cdot\nabla\tilde{\phi}\right]\tilde{u}\right)\notag\\
=&-\frac{2b}{\lambda}\int_0^{+\infty}\sqrt{s}
\int_{\mathbb{R}^3}\Delta\phi\left(\frac{x-\alpha}{A\lambda}\right)|\nabla\tilde{u}_s|^2dxds\notag\\
&+\frac{b}{2A^2\lambda^3}\int_0^{+\infty}\sqrt{s}\int_{\mathbb{R}^3}
\Delta^2\phi\left(\frac{x-\alpha}{A\lambda}\right)|\tilde{u}_s|^2dxds.
\end{align}
where $\tilde{u}_s(t,x):=\sqrt{\frac{2}{\pi}}\frac{1}{-\Delta+s}\tilde{u}(t,x)$, for $s>0$.
Next, we estimate the other term in \eqref{energy-prior-part-2-3D}. Integrating by part as well as the bound \eqref{energy-priori-1-3D}, \eqref{energy-priori-2-3D} and \eqref{energy-priori-3-3D}, we find that
\begin{align}\label{energy-part2-1-3D}
&\Big|-b\Re\left(\int A\nabla\phi(\frac{x-\alpha}{A\lambda})
(f(u)-f(w)-f'(w)\tilde{u})\cdot\overline{\nabla\tilde{u}}\right)\notag\\
&-\frac{1}{2}\frac{b}{\lambda}\Re\left(\int\Delta\phi(\frac{x-\alpha}{A\lambda})
(f(u)-f(w)-f'(w)\tilde{u})\cdot\overline{\tilde{u}}\right)\Big|\notag\\
\lesssim&b\Re\int A\nabla\phi\left(\frac{x-\alpha}{A\lambda}\right)
|\tilde{u}|^{\frac{2}{3}+1}\nabla\tilde{u}dx
+\frac{b}{\lambda}\Re\left(\int\Delta\phi(\frac{x-\alpha}{A\lambda})
(\tilde{u}^{\frac{2}{3}+1}\cdot\overline{\tilde{u}}\right)\notag\\
\lesssim&b\Re\int A\nabla\phi\left(\frac{x-\alpha}{A\lambda}\right)
\nabla(|\tilde{u}|^{\frac{8}{3}})dx
+\frac{b}{\lambda}\Re\left(\int\Delta\phi(\frac{x-\alpha}{A\lambda})
(\tilde{u}^{\frac{2}{3}+1}\cdot\overline{\tilde{u}}\right)\notag\\
\lesssim&\frac{b}{\lambda}\Re\int \Delta\phi\left(\frac{x-\alpha}{A\lambda}\right)
|\tilde{u}|^{\frac{8}{3}}dx\notag\\
\lesssim&\lambda^{-\frac{1}{2}}\|\tilde{u}\|_{\dot{H}^{1/2}}^2
\|\tilde{u}\|_2^{\frac{2}{3}}\lesssim\|\tilde{u}\|_{H^{1/2}}^2.
\end{align}
We consider the other term in \eqref{energy-prior-part-2-3D}. Integrating by parts, we obtain
\begin{align}\label{energy-part2-2-3D}
&-b\Re\left(\int\psi\nabla\phi(\frac{x-\alpha}{A\lambda})
\cdot\overline{\nabla\tilde{u}}\right)-\frac{1}{2}\frac{b}{\lambda}\Re
\left(\int\psi\Delta(\frac{x-\alpha}{A\lambda})\bar{\tilde{u}}\right)\notag\\
&=\Im\left(\int\left[ibA\nabla\phi(\frac{x-\alpha}{A\lambda})\cdot\nabla\psi
+i\frac{b}{2\lambda}\Delta\psi(\frac{x-\alpha}{A\lambda})\psi\right]\bar{\tilde{u}}\right).
\end{align}
Finally, we insert \eqref{energy-part2-4-3D}, \eqref{energy-part2-1-3D} and \eqref{energy-part2-2-3D} into \eqref{energy-prior-part-2-3D}. This yield that
\begin{align}
&\frac{1}{2}\Im\left(\nabla{\tilde{\phi}}\cdot(\nabla\partial_t\tilde{u}\bar{\tilde{u}}
+\nabla\tilde{u}\partial_t\bar{\tilde{u}})\right)\notag\\
=&-\frac{2b}{\lambda}\int_0^{+\infty}\sqrt{s}
\int_{\mathbb{R}^3}\Delta\phi(\frac{x-\alpha}{A\lambda})|\nabla\tilde{u}_s|^2dxds
+\frac{b}{2A^2\lambda^3}\int_0^{+\infty}\sqrt{s}\int_{\mathbb{R}^3}
\Delta^2\phi(\frac{x-\alpha}{A\lambda})|\tilde{u}_s|^2dxds\notag\\
&+\Im\left(\int\left[ibA\nabla\phi(\frac{x-\alpha}{A\lambda})\cdot\nabla\psi
+i\frac{b}{2\lambda}\Delta\psi(\frac{x-\alpha}{A\lambda})\psi\right]\bar{\tilde{u}}\right)\notag\\
&-b\Re\left(\int A\nabla\phi(\frac{x-\alpha}{A\lambda})
(\frac{4}{3}|w|^{\frac{2}{3}}\tilde{u}
+\frac{1}{3}|w|^{-\frac{4}{3}}w^2\bar{\tilde{u}})\cdot\overline{\nabla \tilde{u}}\right)\notag\\
&-\frac{1}{2}\frac{b}{\lambda}\Re\left(\int\Delta\phi(\frac{x-\alpha}{A\lambda})
(\frac{4}{3}|w|^{\frac{2}{3}}\tilde{u}
+\frac{1}{3}|w|^{-\frac{4}{3}}w^2\bar{\tilde{u}})\cdot\overline{\tilde{u}}\right)\notag\\
&+\mathcal{O}\left(\|\tilde{u}\|_{H^{1/2}}^2\right).
\end{align}
This completes the proof of lemma.
\end{proof}

\section{Backwards Propagation of small}
We now apply the energy estimate of the previous section in order to establish a bootstrap argument that will be needed in the construction of minimal mass blowup solution. Let $u=u(t,x)$ be a radial solution to \eqref{equ-1-hf-3D} defined in $[t_0,0)$. Assume that $t_0<t_1<0$ and suppose that $u$ admits on $[t_0,t_1]$ a geometrical decomposition of the form
\begin{align}\label{back-decomposition-3D}
u(t,x)=\frac{1}{\lambda^{\frac{3}{2}}(t)}\left(Q_{\mathcal{P}(t)}+\epsilon\right)
\left(s,\frac{x-\alpha(t)}{\lambda(t)}\right)
e^{i\gamma(t)},
\end{align}
where $\epsilon=\epsilon_1+i\epsilon_2$ satisfies the orthogonality condition \eqref{mod-orthogonality-condition-3D} and $b^2+|\beta|+\|\epsilon\|_{H^{1/2}}^2\ll1$ holds. We set
\begin{align}\label{back-small-part-3D}
\tilde{u}(t,x)=\frac{1}{\lambda^{\frac{3}{2}}(t)}\epsilon
\left(s,\frac{x-\alpha(t)}{\lambda(t)}\right)
e^{i\gamma(t)}.
\end{align}
Define the constant
\begin{align}\label{back-define-1-3D}
A_0=\sqrt{\frac{e_1}{E_0}},
\end{align}
with the constant $e_1=\frac{1}{2}(L_{-}S_{1,0},S_{1,0})>0$ and $E_0=E(u_0)>0$. Define
\begin{align}\label{back-define-2-3D}
B_0=\frac{P_0}{p_1},
\end{align}
where $p_1=2\int_{\mathbb{R}^3}L_{-}S_{0,1}\cdot S_{0,1}>0$ is a constant and $P_0=P(u_0)$. Here we notice that $B_0$ is a vector.

Now we claim that the following backwards propagation estimate holds.
\begin{lemma}\label{lemma-backwards-3D}(\textbf{Backwards propagation of smallness})
Assume that, for some $t_1<0$ sufficiently close to $0$, we have the bounds
\begin{align*}
&\left|\|u\|_2^2-\|Q\|_2^2\right|\leq\lambda^2(t_1),\\
&\|D^{\frac{1}{2}}\tilde{u}(t_1)\|_2^2+\frac{\|\tilde{u}\|_2^2}{\lambda(t_1)}\leq\lambda(t_1),\\
&\left|\lambda(t_1)-\frac{t_1^2}{4A_0^2}\right|\leq\lambda^{\frac{3}{2}}(t_1),\
\left|\frac{b(t_1)}{\lambda^{\frac{1}{2}}(t_1)}\right|\leq\lambda(t_1),\
\left|\frac{\beta(t_1)}{\lambda(t_1)}-B_0\right|\leq\lambda(t_1),
\end{align*}
where $A_0$ and $B_0$ are defined in \eqref{back-define-1-3D} and \eqref{back-define-2-3D}, respectively. Then there exists a time $t_0<t_1$ depending on $A_0$ and $B_0$ such that for all $t\in[t_0,t_1]$, it holds that
\begin{align*}
&\|D^{\frac{1}{2}}\tilde{u}(t)\|_2^2+\frac{\|\tilde{u}\|_2^2}{\lambda(t)}\leq\lambda(t),\\
&\left|\lambda(t)-\frac{t^2}{4A_0^2}\right|\leq\lambda^{\frac{3}{2}}(t),\
\left|\frac{b(t)}{\lambda^{\frac{1}{2}}(t)}-\frac{1}{A_0}\right|\leq\lambda(t),\
\left|\frac{\beta(t)}{\lambda(t)}-B_0\right|\leq\lambda(t).
\end{align*}
\end{lemma}
\begin{proof}
By assumption, we have $u\in C^0([t_0,t_1];H^{1/2}(\mathbb{R}^3))$. Hence, by this continuity and the continuity of the functions $\{\lambda(t),b(t),\beta(t),\alpha(t)\}$, there exists a time $t_0$ such that for all $t\in[t_0,t_1]$ we have the bounds
\begin{align}\label{back-claim-1-3D}
&\|\tilde{u}\|_2^2\leq K\lambda^2(t),\ \|\tilde{u}(t)\|_{H^{1/2}}\leq K\lambda(t),\\\label{back-claim-2-3D}
&\left|\lambda(t)-\frac{t^2}{4A_0^2}\right|\leq K\lambda^{\frac{3}{2}}(t),\
\left|\frac{b(t)}{\lambda^{\frac{1}{2}}(t)}-\frac{1}{A_0}\right|\leq K\lambda(t),\
\left|\frac{\beta(t)}{\lambda(t)}-B_0\right|\leq K\lambda(t),
\end{align}
with some constant $K>0$. We now claim that the bounds stated in this lemma hold on $[t_0,t_1]$, hence improving \eqref{back-claim-1-3D} and  \eqref{back-claim-2-3D} on $[t_0,t_1]$ for $t_0=t_0(A_0)<t_1$ small enough but independent of $t_1$. We divide the proof into the following steps.

$\mathbf{Step~1~Bounds~on~energy~and~L^2-norm}$. We set
\begin{align}
w(x,t)=\tilde{Q}(t,x)=\frac{1}{\lambda^{\frac{3}{2}}(t)}Q_{\mathcal{P}(t)}
\left(\frac{x-\alpha(t)}{\lambda(t)}\right)e^{i\gamma(t)}.
\end{align}
Let $J_A$ be given by above section. Applying lemma \ref{lemma-energy-estimate-3D}, we claim that we obtain the following coercivity estimate:
\begin{align}\label{back-1-1-3D}
\frac{dJ_A}{dt}\geq&\frac{b}{\lambda^2}\int|\tilde{u}|^2
+\mathcal{O}\left(\|\tilde{u}\|_{H^{1/2}}^2
+\|\tilde{u}\|_{H^{1/2}}^\frac{5}{3}
+\|\tilde{u}\|_{H^{1/2}(\mathbb{R}^3)}^{\frac{7}{6}}
+H(t)^{1/2}+K^4\lambda^{\frac{5}{2}}\right).
\end{align}

Assume \eqref{back-1-1-3D} holds.  By the Sobolev embedding and small of $\epsilon$, we deduce the upper bound
\begin{align}\label{back-1-2-3D}
|J_A|\leq\|D^{\frac{1}{2}}\tilde{u}\|_2^2+\frac{1}{\lambda}\|\tilde{u}\|_2^2
\end{align}
Here we use the following inequality
\begin{align}\label{back-1-3-3D}
\left|\Im\left(\int A\nabla\phi\left(\frac{x-\alpha}{A\lambda}\right)
\cdot\nabla\tilde{u}\bar{\tilde{u}}\right)\right|
\leq\|D^{\frac{1}{2}}\tilde{u}\|_2^2+\frac{1}{\lambda}\|\tilde{u}\|_2^2,
\end{align}
where we can see \cite[Lemma F.1]{KLR2013}. Furthermore, due to the proximity of $Q_{\mathcal{P}}$ to $Q$, we derive the lower bound
\begin{align}
J_A=&\frac{1}{2}\int|D^{\frac{1}{2}}\tilde{u}|^2
+\frac{1}{2}\int\frac{|\tilde{u}|^2}{\lambda}
-\int[F(u)-F(w)-F'(w)\cdot\tilde{u}]\notag\\
&+\frac{b}{2}\Im\left(\int A\nabla\phi
\left(\frac{x-\alpha}{A\lambda}\right)\cdot\nabla\tilde{u}\bar{\tilde{u}}\right)\notag\\
\geq&\frac{C_0}{\lambda}\left[\|\epsilon\|_{H^{1/2}}^2-(\epsilon_1,Q)^2\right],
\end{align}
using the orthogonality conditions \eqref{mod-orthogonality-condition-3D} satisfied by $\epsilon$ and the coercivity estimate for the linearized operator $L=(L_{+},L_{-})$. On the other hand, using the conservation of the $L^2-$mass and applying lemma \ref{lemma-mod-1-3D}, we combine the assumed bounds to conclude that
\begin{align}
|\Re(\epsilon,Q_{\mathcal{P}})|\leq \|\epsilon\|_2^2+\lambda^2(t)+
\left|\int|u|^2-\int|Q|^2\right|.
\end{align}
This implies
\begin{align}\label{back-1-4-3D}
(\epsilon_1,Q)^2\leq o(\|\epsilon\|_2^2)+K^4\lambda^{4}(t).
\end{align}
Next, we define
\begin{align}\label{back-define-norm-1-3D}
H(t):=\|D^{\frac{1}{2}}\tilde{u}(t)\|_2^2
+\frac{1}{\lambda(t)}\|\tilde{u}(t)\|_2^2.
\end{align}
By integrating \eqref{back-1-1-3D} in time and using \eqref{back-1-2-3D}, \eqref{back-1-3-3D} and \eqref{back-1-4-3D}, we find
\begin{align*}
H(t)\lesssim& H(t_1)+K^4\lambda^3(t)
+\int_{t}^{t_1}\Big(\|\tilde{u}\|_{H^{1/2}}^2
+\|\tilde{u}\|_{H^{1/2}}^\frac{5}{3}\\
&+\|\tilde{u}\|_{H^{1/2}(\mathbb{R}^3)}^{\frac{7}{6}}
+H(\tau)^{1/2}
+K^{4}\lambda^{\frac{5}{2}}(\tau)\Big)d\tau\\
\lesssim&H(t_1)+K^4\lambda^3(t)
+\int_{t}^{t_1}\left(H(\tau)+H(t\tau)^{5/6}+H(\tau)^{7/12}+H(\tau)^{1/2}\right)d\tau\\
\lesssim&H(t_1)+K^4\lambda^3(t)+\int_{t}^{t_1}
H(\tau)^{1/2}d\tau.
\end{align*}
According to assumption \eqref{back-small-part-3D}, $\|\epsilon\|_{H^{1/2}(\mathbb{R}^3)}\lesssim1$ and the $\lambda\in C^1$, we have $H(t)\lesssim 1$. Hence, there exists $t_0=t_0(A_0)<t_1$ such that
\begin{align*}
\sup_{t\in[t_0,t_1]}H(t)\lesssim& H(t_1)+K^4\lambda^3(t)+\sup_{t\in[t_0,t_1]}H(t)^{1/2}\int_{t}^{t_1}d\tau\\
=&H(t_1)+K^4\lambda^3(t)+\sup_{t\in[t_0,t_1]}H(t)^{1/2}(t_1-t)\\
\leq& H(t_1)+K^4\lambda^3(t)
+\frac{(\sup_{t\in[t_0,t_1]}H(t)^{1/2})^2}{2}+\frac{(t_1-t)^2}{2}.
\end{align*}
Therefore, we deduce that
\begin{align}\notag
\sup_{t\in[t_0,t_1]}H(t)\lesssim2H(t_1)+2K^4\lambda^3(t)+(t_1-t)^2\lesssim\lambda(t),
\end{align}
which closes the bootstrap for \eqref{back-claim-1-3D}.

\textbf{Step~2~Controlling~the~law~for~the~parameters.} Here the treatment is quite close to the corresponding proof of Lemma 4.2 in \cite{Georgiev-Li,KLR2013}, so we can obtain
\begin{align}\notag
\left|\lambda^{\frac{1}{2}}(t)-\frac{t}{2A_0}\right|\leq
\left|\lambda^{\frac{1}{2}}(t_1)-\frac{t_1}{2A_0}\right|+\mathcal{O}(t^3)\leq t^2,
\end{align}
and
\begin{align}\notag
\left|\frac{\beta(t)}{\lambda(t)}-B_0\right|\leq\lambda(t).
\end{align}
This completes the proof of Step 2.

$\mathbf{Step~3~Proof~of~the~coercivity~estimate~ \eqref{back-1-1-3D}}$. Recalling that $w=\tilde{Q}$. Let $\mathcal{K}_A(\tilde{u})$ denote the terms in $\tilde{u}$ on the righthand side in lemma \ref{lemma-energy-estimate-3D}, that is, we have
\begin{align}
\mathcal{K}_A(\tilde{u})&=
\frac{b}{2\lambda}\int\frac{|\tilde{u}|^2}{\lambda}
-\frac{2b}{\lambda}\int_0^{+\infty}\sqrt{s}
\int_{\mathbb{R}^3}\Delta\phi(\frac{x-\alpha}{A\lambda})|\nabla\tilde{u}_s|^2dxds\notag\\
&+\frac{b}{2A^2\lambda^3}\int_0^{+\infty}\sqrt{s}\int_{\mathbb{R}^3}
\Delta^2\phi(\frac{x-\alpha}{A\lambda})|\tilde{u}_s|^2dxds\notag\\
&-b\Re\left(\int A\nabla\phi(\frac{x-\alpha}{A\lambda})
(\frac{4}{3}|w|^{\frac{2}{3}}\tilde{u}
+\frac{1}{3}|w|^{-\frac{4}{3}}w^2\bar{\tilde{u}})\cdot\overline{\nabla \tilde{u}}\right)\notag\\
&-\frac{1}{2}\frac{b}{\lambda}\Re\left(\int\Delta\phi(\frac{x-\alpha}{A\lambda})
(\frac{4}{3}|w|^{\frac{2}{3}}\tilde{u}
+\frac{1}{3}|w|^{-\frac{4}{3}}w^2\bar{\tilde{u}})\cdot\overline{\tilde{u}}\right).
\end{align}
Recalling that the function $\tilde{u}_s=\tilde{u}_s(t,x)$ with the parameter $s>0$ was defined in lemma \ref{lemma-energy-estimate-3D} to be $\tilde{u}_s=\sqrt{\frac{2}{\pi}}\frac{1}{-\Delta+s}\tilde{u}$ and $\tilde{u}=\frac{1}{\lambda^{\frac{3}{2}}}\epsilon(t,\frac{x}{\lambda})$, we now claim that the following estimate holds:
\begin{align}
\mathcal{K}_A(\tilde{u})\geq\frac{C}{\lambda^{\frac{3}{2}}}\int|\epsilon|^2
+\mathcal{O}(\|\tilde{u}\|_{H^{1/2}}^2+K^4\lambda^{\frac{5}{2}}),
\end{align}
where $C>0$ is some positive constant.
Indeed, from the lemma \ref{lemma-mod-2-3D} and the estimate \eqref{back-claim-1-3D} we obtain that
\begin{align}\label{back-mod-estimate-3D}
|\mathbf{Mod}(t)|\lesssim K^2\lambda^2(t).
\end{align}

First, using $\epsilon(t,x)=\lambda^{\frac{3}{2}}\tilde{u}(t,\lambda x+\alpha)$, we estimate the term
\begin{align}
&\left|b\Re\left(\int A\nabla\phi(\frac{x-\alpha}{A\lambda})
(\frac{4}{3}|w|^{\frac{2}{3}}\tilde{u}
+\frac{1}{3}|w|^{-\frac{4}{3}}w^2\bar{\tilde{u}})\cdot\overline{\nabla \tilde{u}}\right)\right|\notag\\
\lesssim&\|\nabla\tilde{\phi}\|_{L^{\infty}}\Re\left(\int
(\frac{4}{3}|w|^{\frac{2}{3}}\tilde{u}
+\frac{1}{3}|w|^{-\frac{4}{3}}w^2\bar{\tilde{u}})\cdot\overline{\nabla \tilde{u}}\right)\notag\\
\lesssim&b\Re\left(\int
(\frac{4}{3}|Q_{\mathcal{P}}|^{\frac{2}{3}}\epsilon
+\frac{1}{3}|Q_{\mathcal{P}}|^{-\frac{4}{3}}Q_{\mathcal{P}}^2\bar{\epsilon})\cdot\overline{\nabla \epsilon}\right)\notag\\
\lesssim&b\|Q_{\mathcal{P}}\|_{L^{\infty}}^{\frac{2}{3}}
\int\epsilon\nabla\epsilon\notag\\
\lesssim&b\lambda\|\tilde{u}\|_{H^{1/2}}^2\lesssim\|\tilde{u}\|_{H^{1/2}}^2,
\end{align}
where in the last step we use the uniform decay estimate $|Q_{\mathcal{P}}|\leq\langle x\rangle^{-4}$. By the definition of $\mathcal{K}_A(\tilde{u})$ and expressing everything in terms, we have
\begin{align*}
\mathcal{K}_A(\tilde{u})\gtrsim&\frac{b}{2\lambda^2}\bigg\{\int_0^{\infty}\sqrt{s}
\int\Delta\left(\frac{x}{A}\right)|\nabla\epsilon_s|^2dxds+\int|\epsilon|^2\notag\\
&-\frac{1}{2A^2}\int_0^{\infty}\sqrt{s}\int\Delta^2\phi\left(\frac{x}{A}\right)
|\epsilon_s|^2dxds\notag\\
-&\int\frac{4}{3}|Q_{\mathcal{P}}|^{\frac{2}{3}}\epsilon^2+
\frac{1}{3}|Q_{\mathcal{P}}|^{-\frac{4}{3}}Q_{\mathcal{P}}^2|\epsilon|^2\bigg\}+
\mathcal{O}(\|\tilde{u}\|_{H^{1/2}}^{\frac{5}{3}}).
\end{align*}

 Furthermore, thanks to Lemma \ref{lemma-app-c-3-3D}, we have
\begin{align}\notag
\left|\frac{1}{A^2}\int_{s=0}^{+\infty}\sqrt{s}\int\Delta^2\phi_{A}|\epsilon_s|^2dxds\right|
\leq\frac{1}{A}\|\epsilon\|_2^2.
\end{align}
Recalling the definitions of $L_{+,A}$ and $L_{-,A}$ in \eqref{app-c-define-1-3D} and \eqref{app-c-define-2-3D}, respectively. We deduce that
\begin{align}\notag
K_{A}(\tilde{u})=\frac{a}{2\lambda^2}\left\{(L_{+,A}\epsilon_1,\epsilon_1)
+(L_{-,A}\epsilon_2,\epsilon_2)+\mathcal{O}\left(\frac{1}{A}\|\epsilon\|_2^2\right)\right\}
+\mathcal{O}(\|\tilde{u}\|_{H^{1/2}}^2).
\end{align}
Next, we recall that $b\sim\lambda^{\frac{1}{2}}$ due to the above. Hence, by lemma \eqref{lemma-app-c-2-3D} and choosing the $A>0$ sufficiently large, we deduce from previous estimates that
\begin{align}
K_A(\tilde{u})\gtrsim\frac{1}{\lambda^{\frac{3}{2}}}\left\{\int|\epsilon|^2-(\epsilon_1,Q)^2\right\}
\gtrsim\frac{1}{\lambda^{\frac{3}{2}}}\int|\epsilon|^2
+\mathcal{O}(\|\tilde{u}\|_{H^{1/2}}^2+K\lambda^{\frac{5}{2}}).
\end{align}

$\mathbf{Step~4~Controlling~the~remainder~terms~in}~\frac{d}{dt}J_A$. We now control the terms that appear in lemma \ref{lemma-energy-estimate-3D} and contain $\psi$. Here we recall that $w=\tilde{Q}$ and \eqref{equ-app2-hf-3D}, which yields
\begin{align*}
\psi&=\frac{1}{\lambda^{\frac{3}{2}+1}}\Big[i(b_s+\frac{1}{2}b^2)\partial_bQ_{\mathcal{P}}
-i(\frac{\lambda_s}{\lambda}+b)\Lambda Q_{\mathcal{P}}
+i\sum_{j=1}^3((\beta_j)_s+b\beta_j)\partial_{\beta_j}Q_{\mathcal{P}}\\
&-i(\frac{\alpha_s}{\lambda}-\beta)\cdot\nabla Q_{\mathcal{P}}+
\tilde{\gamma}_sQ_{\mathcal{P}}+\Phi_{\mathcal{P}}\Big](\frac{x-\alpha}{\lambda})e^{i\gamma}.
\end{align*}
Here $\Phi_{\mathcal{P}}$ is the error term given in lemma \ref{lemma-3app-3D}. In fact, by the estimate for $Q_{\mathcal{P}}$ and $\Phi_{\mathcal{P}}$ from lemma \ref{lemma-3app-3D} and recalling \eqref{back-mod-estimate-3D}, we deduce the rough pointwise estimate
\begin{align}\notag
\left|\nabla^k\psi(x)\right|\lesssim\frac{1}{\lambda^{\frac{5}{2}+k}}
\left\langle\frac{x-\alpha}{\lambda}\right\rangle^{-4} K^2\lambda^2,\,\text{for}\,k=0,1.
\end{align}
Hence
\begin{align}\notag
\|\nabla^k\psi\|_2\lesssim K^2\lambda^{1-k},\,\,\text{for}\,\,k=0,1.
\end{align}
In particular, we obtain the following bounds
\begin{align}\notag
&\lambda\|\psi\|_2^2\lesssim K^4\lambda^3,\\
&\left|\Im\left(\int[ibA\nabla\phi(\frac{x-\alpha}{A\lambda})\cdot\nabla\psi+
i\frac{b}{2\lambda}\Delta\phi(\frac{x-\alpha}{A\lambda})\psi]\bar{\tilde{u}}\right)\right|\notag\\
\lesssim&\lambda^{\frac{1}{2}}\|\nabla\psi\|_2\|\tilde{u}\|_2
+\lambda^{-\frac{1}{2}}\|\psi\|_2\|\tilde{u}\|_2\notag\\
\lesssim&K^2\lambda^{\frac{1}{2}}\|\epsilon\|_2\lesssim o\left(\frac{\|\epsilon\|_2^2}{\lambda^{\frac{3}{2}}}\right)+K^4\lambda^{\frac{5}{2}}.
\end{align}
Write $\psi=\psi_1+\psi_2$ with $\psi_2=\mathcal{O}(\mathcal{P}|\mathbf{Mod}|+b^5)=\mathcal{O}(\lambda^{\frac{5}{2}})$, that is, we denote
\begin{align*}
\psi_1&=\frac{1}{\lambda^{\frac{3}{2}+1}}\Big[-(b_s+\frac{1}{2}b^2)S_{1,0}
-i(\frac{\lambda_s}{\lambda}+b)\Lambda Q-(\beta_s+b\beta)\cdot S_{0,1}\\
&-i(\frac{\alpha_s}{\lambda}-\beta)\cdot\nabla Q+
\tilde{\gamma}_sQ\Big](\frac{x-\alpha}{\lambda})e^{i\gamma}.
\end{align*}
Let us first deal with estimating the contributions coming from $\psi_2$. Indeed, since $b^2+|\beta|\sim\lambda$ we note that $\psi_2=\mathcal{O}(\lambda^{\frac{5}{2}})$ satisfies the pointwise bound
\begin{align}\notag
|\nabla^{k}\psi_2(x)|\leq \frac{1}{\lambda^{\frac{3}{2}+k+1}}
\left\langle\frac{x-\alpha}{\lambda}\right\rangle^{-4} K^2\lambda^{\frac{5}{2}},\ \text{for}\ k=0,1.
\end{align}
Hence
\begin{align}\notag
\|\nabla^k\psi_2\|_2\leq K^2\lambda^{\frac{3}{2}-k},\ \text{for}\ k=0,1.
\end{align}
Therefore, we obtain that
\begin{align}\notag
&\left|\Re\left(\int\left[-D\psi_2-\frac{\psi_2}{\lambda}+
(\frac{1}{3}+1)|w|^{\frac{2}{3}}\psi_2
+\frac{1}{3}|w|^{\frac{2}{3}-2}w^2\bar{\psi_2}\right]\bar{\tilde{u}}\right)\right|\notag\\
\leq&\left(\|\nabla\psi_2\|_2+\lambda^{-1}\|\psi_2\|_2+\|w\|_2^{\frac{1}{3}}
\|\psi_2\|_{6}\right)\|\epsilon\|_2\notag\\
\leq&\left(\|\nabla\psi_2\|_2+\lambda^{-1}\|\psi_2\|_2+\|w\|_2^{\frac{2}{3}}
\|\nabla\psi_2\|_2\right)\|\epsilon\|_2\notag\\
\leq&K^2\lambda^{\frac{1}{2}}\|\epsilon\|_2\leq o\left(\frac{\|\epsilon\|_2^2}{\lambda^{\frac{3}{2}}}\right)+K^4\lambda^{\frac{5}{2}},
\end{align}
which is acceptable. Here we used the H\"{o}lder inequality and Sobolev inequality.

Finally, we estimate the terms that contain $\psi_1$ have the same bounded. Indeed, using \eqref{back-mod-estimate-3D} once again and $|\mathcal{P}|\leq\lambda^{\frac{1}{2}}$, as well as $(\epsilon_2,L_-S_{1,0})=(\epsilon_2,\Lambda Q)=\mathcal{O}(\mathcal{P}\|\epsilon\|_2)$ and $(\epsilon_2,L_-S_{0,1})=-(\epsilon_2,\nabla Q)=\mathcal{O}(\mathcal{P}\|\epsilon\|_2)$, thanks to the orthogonality conditions for $\epsilon$, we find the following bound
\begin{align}\notag
&\left|\Re\left(\int\left[-D\psi_1-\frac{\psi_1}{\lambda}+
(\frac{1}{3}+1)|w|^{\frac{2}{3}}\psi_1
+\frac{1}{3}|w|^{\frac{2}{3}-2}w^2\bar{\psi_1}\right]\bar{\tilde{u}}\right)\right|\notag\\
\leq&\frac{|\mathbf{Mod}(t)|}{\lambda^2}\left[|(\epsilon_2,L_-S_{1,0})|+|(\epsilon_2,L_-S_{0,1})|
+|(\epsilon_2,L_-Q)|+\mathcal{O}(\mathcal{P}\|\epsilon\|_2)\right]\notag\\
&+\frac{1}{\lambda^2}\left|\frac{\lambda_s}{\lambda}+b\right||(\epsilon_1,L_+\Lambda Q)|+\frac{1}{\lambda^2}\left|\frac{\alpha_s}{\lambda}-\beta\right||(\epsilon_1,L_+\nabla Q)|\notag\\
\leq&K^2\lambda^{\frac{1}{2}}\|\epsilon\|_2
+\frac{K^2\lambda\|\epsilon\|_2+\lambda^{\frac{5}{2}}}{\lambda^2}
(K\lambda^{\frac{1}{2}}\|\epsilon\|_2+K^2\lambda^{2})\notag\\
\leq&o\left(\frac{\|\epsilon\|_2^2}{\lambda^{\frac{3}{2}}}\right)+K^4\lambda^{\frac{5}{2}},
\end{align}
 Moreover, we also used that $L_+\nabla Q=0$ and $L_+\Lambda Q=-Q$ together with the improved bound in Lemma \ref{lemma-mod-2-3D}, combined with the fact that $|(\epsilon_1,Q)|\leq\lambda^{\frac{1}{2}}\|\epsilon\|_2+K^2\lambda^2$, which follows from $\|\epsilon\|_2\leq\lambda$ and the conservation of $L^2$-norm.
 And the proof of this lemma is complete.
\end{proof}

\section{Existence of minimal mass blowup solutions}
In this section, we prove the following result.
\begin{theorem}
Let $\gamma_0\in\mathbb{R}$, $x_0\in\mathbb{R}^3,P_0\in\mathbb{R}^3$ and $E_0>0$ be given. Then there exist a time $t_0<0$ and a radial solution $u\in C^0([t_0,0); H^{\frac{1}{2}}(\mathbb{R}^3))$ of \eqref{equ-1-hf-3D}  such that $u$ blowup  at time $T=0$ with
\begin{align}\notag
E(u(t))=E_0,\ P(u(t))=P_0,\ \text{and}\ \|u(t)\|_2^2=\|Q\|_2^2.
\end{align}
Furthermore, we have $\|D^{\frac{1}{2}}u(t)\|_2\sim t^{-\frac{1}{4}}$ as $t\rightarrow0^{-}$, and $u$ is of the form
\begin{align}\notag
u(t,x)=\frac{1}{\lambda^{\frac{3}{2}}(t)}[Q_{\mathcal{P}(t)}+\epsilon]
\left(t,\frac{x-\alpha}{\lambda}\right)e^{i\gamma(t)}=\tilde{Q}+\tilde{u},
\end{align}
where $\mathcal{P}(t)=(b(t),\beta(t))$, and $\epsilon$ satisfies the orthogonality condition \eqref{mod-orthogonality-condition-3D}. Finally, the following estimate hold:
\begin{align*}
&\|\tilde{u}\|_2\lesssim\lambda,\,\,\|\tilde{u}\|_{H^{1/2}}\lesssim\lambda^{\frac{1}{2}},\\
&\lambda(t)-\frac{t^2}{4A_0^2}=\mathcal{O}(\lambda^2),\,\,
\frac{b}{\lambda^{\frac{1}{2}}}-\frac{1}{A_0}=\mathcal{O}(\lambda),\,\,
\frac{\beta}{\lambda}(t)-B_0=\mathcal{O}(\lambda),\\
&\gamma(t)=-\frac{4A_0^2}{t}+\gamma_0+\mathcal{O}(\lambda^\frac{1}{2}),\,\,
\alpha(t)=x_0+\mathcal{O}(\lambda^\frac{3}{2}).
\end{align*}
Here $A_0>0$ is the constant defined in \eqref{back-define-1-3D}  and $B_0$ defined in \eqref{back-define-2-3D}.
\end{theorem}
\begin{proof}
Following ideas in \cite{KLR2013,Georgiev-Li}, we can easily obtain this result.  Here we omit details.
\end{proof}

\appendix
\section{Appendix}\label{Appendix}
In this section, we will give some lemmas, which are very important but the proof is relatively simple. By standard arguments as \cite{KLR2013,Georgiev-Li}, we can obtain the following results. Here we omit the details.

In the following, we assume that $A>0$ is a sufficiently large constant. Let $\phi:\mathbb{R}\rightarrow\mathbb{R}$ be the smooth cutoff function introduced in Section \ref{section-refined-energy-3D}. For $\epsilon=\epsilon_1+i\epsilon_2\in H^{1/2}(\mathbb{R}^3)$, we consider the quadratic forms
\begin{align}\label{app-c-define-1-3D}
L_{+,A}(\epsilon_1):&=\int_{s=0}^{\infty}\sqrt{s}\int\Delta \phi_A|\nabla(\epsilon_1)_{s}|^2dxds+\int|\epsilon_{1}|^2-\frac{3}{8}\int Q^{\frac{2}{3}}|\epsilon_1|^2\\\label{app-c-define-2-3D}
L_{-,A}(\epsilon_2):&=\int_{s=0}^{\infty}\sqrt{s}\int\Delta \phi_A|\nabla(\epsilon_2)_{s}|^2dxds+\int|\epsilon_{2}|^2-\int Q^{\frac{2}{3}}|\epsilon_2|^2,
\end{align}
where $\Delta \phi_A=\Delta(\phi(\frac{x}{A}))$. As in lemma \ref{lemma-energy-estimate-3D}, we denote
\begin{align}\label{app-c-define-3-3D}
u_s=\sqrt{\frac{2}{\pi}}\frac{1}{-\Delta+s}u,\  \ \text{for}\ s>0.
\end{align}
We start with the following simple identity.

For $u\in H^{1/2}(\mathbb{R}^3)$, we have
\begin{align}\label{app-c-identity-3D}
\int_0^{\infty}\sqrt{s}\int_{\mathbb{R}^3}|\nabla u_s|^2dxds=\|D^{1/2}u\|_2^2.
\end{align}
Indeed, by applying Fubini's theorem and using Fourier transform, we find that
\begin{align}\notag
\int_0^{\infty}\sqrt{s}\int_{\mathbb{R}^3}|\nabla u_s|^2dxds=\frac{2}{\pi}
\int_{\mathbb{R}^3}\int_0^{\infty}\frac{\sqrt{s}}{(\xi^2+s)^2}|\xi|^2|\hat{u}(\xi)|^2d\xi
=\|D^{1/2}u\|_2^2.
\end{align}
In particular, we have
\begin{align}\label{app-c-identity-2-3D}
\frac{2}{\pi}\int_0^{\infty}\sqrt{s}\int_{\mathbb{R}^3}|D^{\alpha}u_s|^2dxds
=\|D^{\alpha-\frac{1}{2}}u\|_2^2.
\end{align}
Next, we establish a technical result, which show that, when taking the limit $A\rightarrow+\infty$, the quadratic form $\int_0^{\infty}\sqrt{s}\int\Delta\phi_A|\nabla u_s|^2dxds+\|u\|_2^2$ defines a weak topology that serves as a useful substitute for weak convergence in $H^{1/2}(\mathbb{R}^3)$. The precise statement reads as follows.
\begin{lemma}\label{app-lemma-c-1-3D}
Let $A_n\rightarrow\infty$ and suppose that $\{u_n\}_{n=1}^{\infty}$ is a sequence in $H^{1/2}(\mathbb{R}^3)$ such that
\begin{align}\notag
\int_0^{\infty}\sqrt{s}\int\Delta\phi_{A_n}|\nabla (u_n)_s|^2dxds+\|u_n\|_2^2\leq C,
\end{align}
for some constant $C>0$ independent of $n$. Then, after possibly passing to a subsequence of $\{u_n\}_{n=1}^{\infty}$, we have that
\begin{align}\notag
u_n\rightharpoonup u\ \text{weakly in}\ L^2(\mathbb{R}^3)\ \text{and}\ u_n\rightarrow u\ \text{strongly in}\ L^2_{loc}(\mathbb{R}^3),
\end{align}
and $u\in H^{1/2}(\mathbb{R}^3)$. Moreover, we have the bound
\begin{align}\notag
\|D^{1/2}u\|_2^2\leq\liminf_{n\rightarrow\infty}
\int_0^{\infty}\sqrt{s}\int\Delta\phi_{A_n}|\nabla (u_n)_s|^2dxds.
\end{align}
\end{lemma}

\begin{lemma}\label{lemma-app-c-2-3D}
Let $L_{+,A}(\epsilon_1)$ and $L_{-,A}(\epsilon_2)$ be the quadratic forms defined in \eqref{app-c-define-1-3D} and \eqref{app-c-define-2-3D}, respectively. Then there exist a constant $C_0>0$ and $A_0>0$ such that, for all $A\geq A_0$ and all $\epsilon=\epsilon_1+i\epsilon_2\in H^{1/2}(\mathbb{R}^3)$, we have the coercivity estimate
\begin{align}\notag
(L_{+,A}\epsilon_1,\epsilon_1)+(L_{-,A}\epsilon_2,\epsilon_2)\geq& C_0\int|\epsilon|^2-\frac{1}{C_0}\Big\{(\epsilon_1,Q)^2\notag\\
&+(\epsilon_1,S_{1,0})^2
+|(\epsilon_1,S_{0,1})|^2+|(\epsilon_2,\rho_1)|^2\Big\}.
\end{align}
Here $S_{1,0}$ and $S_{0,1}$ are the unique functions such that $L_{-}S_{1,0}=\Lambda Q$ with $(S_{1,0}, Q)=0$ and $L_{-}S_{0,1}=-\nabla Q$ with $(S_{0,1},Q)=0$, respectively, and the function $\rho_1$ is defined in \eqref{mod-definition-rho-3D}.
\end{lemma}

\begin{lemma}\label{lemma-app-c-3-3D}
For any $u\in L^2(\mathbb{R}^3)$, we have the bound
\begin{align}\notag
\left|\int_{s=0}^{+\infty}\sqrt{s}\int\Delta^2\phi_{A}|u_s|^2dxds\right|
\leq\frac{1}{A}\|u\|_2^2.
\end{align}
\end{lemma}

\begin{lemma}\label{lemma-app-coercivity-estimate-3D}
Let $L_{+}\epsilon_1$ and $L_{-}\epsilon_2$ be the  defined as section \ref{section2-3D}, respectively. Then there exist a constant $C_0>0$ and $A_0>0$ such that, for all $A\geq A_0$ and all $\epsilon=\epsilon_1+i\epsilon_2\in H^{1/2}(\mathbb{R}^3)$, we have the coercivity estimate
\begin{align}\notag
(L_{+}\epsilon_1,\epsilon_1)+(L_{-}\epsilon_2,\epsilon_2)\geq& C_0\int|\epsilon|^2-\frac{1}{C_0}\Big\{(\epsilon_1,Q)^2\notag\\&+(\epsilon_1,S_{1,0})^2
+|(\epsilon_1,S_{0,1})|^2+|(\epsilon_2,\rho_1)|^2\Big\}.
\end{align}
Here $S_{1,0}$ and $S_{0,1}$ are the unique functions such that $L_{-}S_{1,0}=\Lambda Q$ with $(S_{1,0}, Q)=0$ and $L_{-}S_{0,1}=-\nabla Q$ with $(S_{0,1},Q)=0$, respectively, and the function $\rho_1$ is defined in \eqref{mod-definition-rho-3D}.
\end{lemma}

\textbf{Acknowledgements}

VG is partially supported by Project 2017``Problemi stazionari e di evoluzione nelle equazioni di campo nonlineari" of INDAM, GNAMPA - Gruppo Nazionale per l'Analisi Matematica, la Probabilita e le loro Applicazioni, by Institute of Mathematics and Informatics, Bulgarian Academy of Sciences, by Top Global University Project, Waseda University and the Project PRA 2018 49 of University of Pisa. YL was supported by the China Scholarship Council (201906180041).


\begin{thebibliography}{99}
\bibitem{Abramowita-1964} M. Abramowitz, I. A. Stegun, Handbook of mathematical functions with formulas, graphs, and mathematical tables. National Bureau of Standards Applied Mathematics Series 55, US Government Printing Office, Washington, DC, 1964.
\bibitem{BGV2018}J. Bellazzini, V. Georgiev, N. Visciglia, Long time dynamics for semirelativistic NLS and half wave in arbitrary dimension, Math. Ann. 371 (2018) 707-740.
\bibitem{BGV} J. Bellazzini, V. Georgiev, N. Visciglia, Traveling waves for the quartic focusing half wave equation in one space dimension, (2018) https://arxiv.org/abs/1804.07075
\bibitem{BGLV2019} J. Bellazzini, V. Georgiev, E. Lenzmann, N. Visciglia, On traveling solitary waves and absence of small data scattering for nonlinear half-wave equation, Commun. Math. Phys. 372, (2019) 713-732.

\bibitem{lenzmann-2016blowup}T. Boulenger, D. Himmelsbach, E. Lenzmann, Blowup for fractional NLS, J.  Funct. Anal. 271 (2016) 2569-2603.
\bibitem{majda2001}D. Cai, A. J. Majda, D. W. McLaughlin, and E. G. Tabak, Dispersive wave turbulence in one dimension, Phys. D 152 (2001), 551-572.
\bibitem{Choffrut-2018}A. Choffrut, O. Pocovnicu,  Ill-posedness of the cubic nonlinear half-wave equation and other fractional NLS on the real line. Int. Math. Res. Not. IMRN (2018), no. 3, 699-738.

\bibitem{cho2013}Y. Cho, H. Hajaiej, G. Hwang, and T. Ozawa, On the Cauchy problem of fractional Sch\"{o}dinger equation with Hartree type nonlinearity, Funkcial. Ekvac. 56 (2013), no. 2, 193-224.
\bibitem{Cote-Coz-2011} R. C\^{o}te, S. Le Coz, High-speed excited multi-soliton in nonlinear Schr\"{o}dinger equations, J. Math. Pures Appl. 96 (2011) 135-166.
\bibitem{eckhaus1983}W. Eckhaus, P. Schuur, The emergence of solitons of the Korteweg-de Vries equation from arbitrary initial conditions, Math. Methods Appl. Sci. 5 (1983), no. 1, 97-116.

\bibitem{Elgart2007}A. Elgart, B. Schlein,  Mean field dynamics of boson stars. Commun. Pure Appl. Math. 60(4), 500-545 (2007)

\bibitem{frank-lenzmann2013}R. Frank, E. Lenzmann, Uniqueness of non-linear ground states for fractional Laplacians in $\mathbb{R}$, Acta Math. 210 (2013), no. 2, 261-318.

\bibitem{FrankLS2016}R. L. Frank,  E. Lenzmann, L. Silvestre, Uniqueness of radial solutions for the fractional Laplacian. Comm. Pure Appl. Math. 69 (2016),  1671-1726.
\bibitem{FJL2007}J. Fr\"{o}hlich, B. L. G. Jonsson, E. Lenzmann,  Boson stars as solitary waves. Commun. Math. Phys. 274 (2007) 1-30.
\bibitem{Georgiev-Li} V. Georgiev, Y. Li, Nondispersive solutions to the mass critical half-wave equation in two dimensions, https://arxiv.org/abs/2007.15370.
\bibitem{Georgiev-Visciglia2016}V. Georgiev, N.  Tzvetkov, N. Visciglia, On the regularity of the flow map associated with the 1D cubic periodic half-wave equation. Differential Integral Equations 29 (2016), no. 1-2, 183-200.

\bibitem{Lenzmann-Raphael-2018} P. G\'{e}rard, E. Lenzmann, O. Pocovnicu, P. Rapha\"{e}l, A two-soliton with transient turbulent regime for the cubic half-wave equation on the real line. Ann. PDE 4 (2018), 7
     https://doi.org/10.1007/s40818-017-0043-7.
\bibitem{Ionescu2014}A. D. Ionescu and F. Pusateri, Nonlinear fractional Schr\"{o}dinger equations in one dimension, J. Funct. Anal. 266 (2014), no. 1, 139-176.

\bibitem{k-lenzmann2013}K. Kirkpatrick, E. Lenzmann, and G. Staffilani, On the continuum limit for discrete NLS with long-range lattice interactions, Commun. Math. Phys. 317 (2013), no. 3, 563-591.

\bibitem{klein2014}C. Klein, C. Sparber, and P. Markowich, Numerical study of fractional nonlinear Schr\"{o}dinger equations, Proc. R. Soc. A 470 (2014), no. 2172, 20140364.
\bibitem{Raphael-2009-cpam} J. Krieger, Y. Martel, P. Rapha\"{e}l, Two-soliton solutions to the three-dimensional gravitational Hartree equation, Commun. Pure Appl. Math. 62 (11) (2009) 1501-1550.
\bibitem{KLR2013}J. Krieger, E. Lenzmann, P. Rapha\"{e}l, Nondispersive solutions to the $L^2$-critical half-wave equation, Arch. Ration. Mech. Anal. 209  (2013) 61-129.
\bibitem{Li-Zhao-2020} Y. Li, D. Zhao, Q. X. Wang, Existence of the stable traveling wave for half-wave equation with $L^2$-critical combined nonlinearities, Appl. Anal. (2020)  https://doi.org/10.1080/00036811.2020.1811976
\bibitem{majda1997}A. J. Majda, D. W. McLaughlin, and E. G. Tabak, A one-dimensional model for dispersive wave turbulence, J. Nonlinear Sci. 7 (1997), no. 1, 9-44.
\bibitem{MerleR2006} F, Merle, P, Rapha\"{e}l,  On a sharp lower bound on the blow-up rate for the $L^2$ critical nonlinear Schr\"{o}dinger equation. J. Amer. Math. Soc. 19 (2006),  37-90.
\bibitem{Merle-Raphael-Duke2014} F. Merle, P. Rapha\"{e}l, J. Szeftel, On collapsing ring blow-up solutions to the mass supercritical nonlinear Schr\"{o}dinger equation. Duke Math. J. 163 (2014) 369-431.
\bibitem{Ozawa-Visciglia2016} T. Ozawa, N. Visciglia,  An improvement on the Br\'{e}zis-Gallou\"{e}t technique for 2D NLS and 1D half-wave equation. Ann. Inst. H. Poincar\'{e} Anal. Non Lin\'{e}aire 33 (2016), no. 4, 1069-1079.

\bibitem{Raphael2011-Jams} P. Rapha\"{e}l, J. Szeftel, Existence and uniqueness of minimal blow-up solutions to an inhomogeneous mass critical NLS. J. Amer. Math, Soc. 24 (2) (2011), 471-546.

%\bibitem{stein-1993} E. M. Stein, Harmonic analysis: real-variable methods, orthogonality, and oscillatory intergral. Princeton Mathematical series, vol. 43, Princeton University Press 1993.

\bibitem{Weinstein1987}M. I. Weinstein, Existence and dynamic stability of solitary wave solutions of equations arising in long wave propagation, Comm. Partial Differential Equations 12 (1987), no. 10, 1133-1173.

\end{thebibliography}
\end{document}